\documentclass{amsart}
 \usepackage{amsmath}
\usepackage{amssymb}
\usepackage{amsfonts}
      \makeatletter
     \def\section{\@startsection{section}{1}%
     \z@{.7\linespacing\@plus\linespacing}{.5\linespacing}%
     {\bfseries
     \centering
     }}
     \def\@secnumfont{\bfseries}
     \makeatother
       \usepackage{a4wide}

   \newtheorem{theorem}{Theorem}[section]
\newtheorem{lemma}[theorem]{Lemma}
\newtheorem{corollary}[theorem]{Corollary}
\newtheorem{proposition}[theorem]{Proposition}

\theoremstyle{definition}

\newtheorem{remark}[theorem]{Remark}
\newtheorem{remarks}[theorem]{Remarks}

\numberwithin{equation}{section}

\def \a{{\alpha}}
\def \b{{\beta}}
\def \D{{\Delta}}

\def \d{{\delta}}
\def \e{{\varepsilon}}

\def \g{{\gamma}}

\def \o{{\omega}}
\def \O{{\Omega}}

\newcommand{\dist}{\stackrel{\mathcal{D}}{=}}

\def \A{{\mathcal A}}

\def \E{{\bf E}\, }

\def \N{{\bf N}}

\def \P{{\bf P}}

\def \qq{{\qquad}}
\def \R{{\bf R}}

\def \Z{{\bf Z}}

\def \dd{{\rm d}}

\def \noi{{\noindent}}

\def\E{{\mathbb E \,}}

\def\P{{\mathbb P}}

\def\R{{\mathbb R}}

\def\Z{{\mathbb Z}}

\def\N{{\mathbb N}}

 \scrollmode
 
   \font\sevenrm= cmr10 at 7 pt
     \scrollmode
  
\def\ddate {\sevenrm \ifcase\month\or January\or
February\or March\or April\or May\or June\or July\or
August\or September\or October\or November\or December\fi\! {\the\day}, \!{\sevenrm\the\year}}

\scrollmode

\title[$\boldsymbol{\Phi}$-variation of stochastic processes]{On the 
 $\boldsymbol{\Phi}$-variation of stochastic processes with exponential moments}
   
\begin{document}
\author{Andreas Basse-O'Connor}
  \author{Michel   WEBER}
  \address{Department of Matematics,  Aarhus University, Ny Munkegade 116, DK-8000 Aarhus C,   Denmark.
 E-mail:    {\tt  basse@math.au.dk}}
 \address{IRMA, Universit\'e
Louis-Pasteur et C.N.R.S.,   7  rue Ren\'e Descartes, 67084
Strasbourg Cedex, France.
   E-mail:    {\tt  michel.weber@math.unistra.fr}}
\footnote{  
\par 
\ddate{}}

\begin{abstract} We obtain  sharp sufficient conditions for  exponentially integrable stochastic processes 
$X=\{X(t)\!\!: t\in [0,1]\}$,  
   to have   sample paths with bounded $\Phi$-variation.  When $X$ is moreover Gaussian, we also provide  a  bound   of  the expectation  of the associated $\Phi$-variation norm of $X$. For an  Hermite process $X$  of order  $m\in \N$ and  of Hurst index $H\in (1/2,1)$, we show that $X$ is of bounded $\Phi$-variation where $\Phi(x)=x^{1/H}(\log(\log 1/x))^{-m/(2H)}$, and  that this $\Phi$ is optimal. This shows  that in terms of $\Phi$-variation, the Rosenblatt process (corresponding to $m=2$) has  more rough sample paths than the fractional Brownian motion (corresponding to $m=1$). 
   
   \end{abstract}

\maketitle
 \section{Introduction and Main Results.}\label{sec1}

Let $\Phi\!\!: \R_+\to \R_+$ be a  strictly increasing  continuous   function such that $\Phi(0)=0$ and $\lim_{t\to \infty} \Phi(t)=\infty$.
  The $\Phi$-variation of a      function
$f:[0,1] \to \R$    is defined according to Young~\cite{Y} by 
\begin{equation*}\label{phivar}
 \mathcal   V_\Phi(f):= \sup_{0=t_0<\dots<t_n=1 \atop   n\in \N}\sum_{i=1}^{n}
 \Phi\big(|f(t_{i})-f(t_{i-1})|\big) .
 \end{equation*}
The particular case $\Phi(t)= t^p$ is  classical and corresponds for $p=1$ to the concept of bounded variation introduced by Jordan~\cite{J}, and for $p>1$ to the one of bounded $p$-variation introduced by Wiener \cite{Wi}. The $\Phi$-variation is closely related to convergence of Fourier series (see \cite[Chapter~11]{Dud-Nor}), 
Hausdorff dimension (see e.g.\ \cite[Theorem~8.4]{Bl-Ge}), rough paths theory (see  \cite[Definition~9.15]{Friz}) and  integration theory (see \cite[Chapter~3]{Dud-Nor}). 
 If  $d$ is a pseudo-metric on  $[0,1]$ ($d$ has all properties of a metric except for the implication $d(s,t)=0\Rightarrow s=t$), we will  write
\begin{equation*}\label{phivar-1}
 \mathcal   V(\Phi,d):= \sup_{0=t_0<\dots<t_n=1 \atop   n\in \N}\sum_{i=1}^{n}
 \Phi\big(d(t_i,t_{i-1})\big) .
 \end{equation*}
Furthermore, set $\log^*(x)=\log(1+x)$ and  $\log^*_2(x)=\log^*(\log^*x)$ for all $x\geq 0$. In several cases     we only define a given function $\Phi\!:\R_+\to\R$ explicitly on $(0,\infty)$, and in this case we always set $\Phi(0)=0$.

A classical result by P.\ L\'evy states that the sample paths of a Brownian motion are of bounded $p$-variation if and only if $p>2$. This result has been improved by 
Taylor~\cite[Theorem~1 and its Corollary]{Taylor}  who derived  that  the optimal  $\Phi$-variation function of the Brownian motion is $\Phi(t)=t^2/\log_2^*(1/t)$. Taylor's result has been extended to the fractional Brownian motion by Dudley and Norvai{\v{s}}a~\cite{Dud-Nor} and to   other Gaussian processes with stationary increments by  
Kawada and K\^ono~\cite{Ka-Ko} and Marcus and Rosen~\cite{Mar-Ros, Mar-Ros-pa-1}. Furthermore, the ability to  solve rough differential equations for the Brownian motion under minimal regularity assumptions relies on the above mentionned  characterization of the Brownian sample paths of Taylor, see Theorems~10.41 and 13.15 in \cite{Friz}.
 \vskip 1 pt Let us recall and briefly discuss a quite general result by Jain and Monrad  which 
has also motivated this work,  and which is the  key 
ingredient  in the recent  work Friz et al.~\cite{Friz-JM} on rough path analysis of Gaussian processes.  Assume that $X=\{X(t)\!: t\in [0,1]\}$ is a  centered  Gaussian process and let 
$d(s,t)=\|X(s) -X(t)\|_{L^2}$ for all $s,t\in [0,1]$. Let $1<p<\infty$ (the case $p=1$ being trivial).  Jain and Monrad~\cite[Theorem~3.2]{JaMo} showed that if
\begin{equation}\label{eq:JaMo}
\mathcal V(\Psi,d)<\infty\qquad \text{where}\qquad \Psi(t)=t^p(\log_2^*(1/t))^{p/2}, \quad t>0
\end{equation}
then $X$ has sample paths of bounded $p$-variation almost surely.
A closer examination of their proof shows that \eqref{eq:JaMo}
      implies the Lipschitz property
\begin{equation}\label{eq23}
| X(t,\omega)-X(s,\omega)| \le  C(\omega)  d(s,t)\Big(\log^* \frac{1}{d(s,t )} \Big)^{1/2}  \qquad \forall 0\le s,t\le 1
\end{equation}
with probability one, which in turn  is shown under the  weaker   and also  necessary  
assumption   
\begin{equation}\label{eq:discussJaMo}
\mathcal V(\tilde \Psi,d)<\infty\qquad \text{where} \qquad \tilde \Psi(t)=t^p,\quad t\geq 0. 
\end{equation}
 As  (\ref{eq23}) immediately  implies that $X $  has almost all sample paths of bounded $\Phi$-variation, where 
$\Phi(t) = t^p (\log^*(1/t))^{-p/2}$, $t>0$,  the double
logarithm term in  $\Psi$  indicates which strengthening of the assumption 
\eqref{eq:discussJaMo} is necessary to get the finer property of having paths of bounded $p$-variation. 
   It also follows (as remarked,  \eqref{eq:discussJaMo} implies \eqref{eq23}, see notably Remark \ref{entropy}) that the latter property holds only if $X$ has the Lipschitzian property \eqref{eq23}. For $d$    continuous with respect to the usual metric on $[0,1]$, the fact that $X$ be of bounded $p$-variation  almost surely, thus implies that $X$ is   continuous almost surely for the usual metric, which is {\it in contrast} with the fact that a function  $f$  with bounded $p$-variation is   
not necessarily  continuous (although both limits
$\lim_{y\to x_-}    f(y)$, $\lim_{y\to x_+}  
f(y)$   exist).  
Recall in addition that $f$ is     continuous almost everywhere  in the sense of the Lebesgue measure, see Bruneau \cite{Br2} for  these facts; and that by    Jordan's   example  \cite{J},    a function with finite variation  may have positive jumps on each rational.

 \vskip 4 pt We  consider  exponentially integrable processes and study the $\Phi$-variation  of their sample paths. We obtain   general sufficient conditions for the sample paths to be of
bounded
$\Phi$-variation.  Our conditions are sharp. When applied to Gaussian processes, we recover and complete  Jain and Monrad   sufficient condition
for bounded $p$-variation, but also extend it to general  $\Phi$-variation spaces. 

\vskip 2 pt Introduce for every $0<\a<\infty$,  the functions
$\phi_\a(x)=e^{|x|^\alpha}-1$, $x\in \R$.
  We consider processes $X$ satisfying the following increment condition, in which $d$ is a given pseudo-metric on $[0,1]$: 
 \begin{equation}\label{eq:673}
\hbox{For some    $0<\a<\infty$,} \qq \E\phi_\a\Big(\Big| \frac{X(t)-X(s)}{ d(t,s)}\Big|\Big)\leq 2 \qquad \text{for all }t,s\in [0,1].
\end{equation}
The cases  $1\le \a<\infty$ and $0<\a<1$ are of different nature since in the second case the functions $\phi_\a$ are no longer convex. Accordingly,  define for a  random variable $U$
 \begin{equation*}
 \|U\|_{\phi_\a}=\inf\big\{\Delta>0: \E \phi_\a\Big(\frac{U}{\Delta}\Big)\leq 1\big\} 
 \end{equation*}
and  write $U\in L^{\phi_\a}$, if  $\|U\|_{\phi_\a}<\infty$. 
For  $\a \geq 1$, $\phi_\a$  is an Orlicz function and the space $(L^{\phi_\a},\|\cdot\|_{\phi_\a})$ is an Orlicz space and hence a Banach space.  
However, for $0<\a<1$, $(L^{\phi_\a},\|\cdot\|_{\phi_\a})$ is only  a quasi-Banach space, i.e.\ satisfies all the axioms of a Banach space except that the triangle inequality is replaced by: there exists a constant $K_\a>0$ such that 
\begin{equation*}\label{def-K-a}
\|U+Y\|_{\phi_\a}\leq K_\a\big(\|U\|_{\phi_\a}+\|Y\|_{\phi_\a}\big)\qquad \text{for all }U,Y\in L^{\phi_\a}
\end{equation*} 
(we may  choose $K_\a=2^{1/\a}$ in our case). We refer to  \cite{Rao-Ren} for the theory of Orlicz spaces. The quasi-Banach space property of $(L^{\phi_\alpha}, \|\cdot\|_{\phi_\a})$  for $0<\a<1$ follows by \cite[Section~2.1]{MaOr}. Furthermore, we refer to  Talagrand~\cite{Talagrand} for sharp results on sample boundedness 
of  stochastic processes satisfying \eqref{eq:673}.

\smallskip 
The following theorem provides sufficient conditions for a class of stochastic  processes with finite exponential moments to have sample paths of bounded $\Phi$-variation.

\begin{theorem}\label{thm-alpha}
Let $X=\{X(t)\!: t\in [0,1]\}$ be a separable stochastic process satisfying the increment condition \eqref{eq:673} for some $\alpha>0$.
 Suppose that there exists   a continuous, strictly increasing  function $\Psi: \R_+\to \R_+$ with
$\mathcal V(\Psi,d)<\infty$.
  Suppose moreover that there exists $C>0$ such that for   all $M>0$,
 \begin{align}\label{con-1-thm1-alpha}
 {}&   \sum_{m=0}^\infty   2^{ m}   \Phi( M y_m)e^{- x_m^\alpha} <\infty   \qquad 
    \   where \\ \label{con-1-thm1-alpha-2}
  {}&  x_m =\frac{\Phi^{-1}( C 2^{-m})}{
K_\a\int_0^{2^{-m}} (\log^*(\frac{ 2^{-m}}{\e}))^{1/\alpha}  \, \Psi^{-1}(\dd \e)},  
\qquad \quad y_m =\int_{0}^{2^{-m}} \Big(\log^*\Big(\frac{1}{\e}\Big)\Big)^{1/\alpha}\, \Psi^{-1}(\dd \e),
 \end{align} 
 and $K_\a$ is universal constant  only depending on $\a$. Then $X$ has sample paths of bounded $\Phi$-variation almost surely. 
  \end{theorem}

 Here and in the next theorems, separability is understood with respect to the usual metric on $[0,1]$, see 
 Subsection~\ref{sec-pre}. For $\a=2$,  a process  satisfying \eqref{eq:673} is called sub-Gaussian and this case includes all Gaussian processes if we set $d(s,t)= \|X(s)-X(t)\|_{L^2}$ for all $s,t\in [0,1]$.   On the other hand, processes satisfying \eqref{eq:673} includes many non-Gaussian  processes, for example all processes which are represented by multiple Wiener--It\^o integrals of a fixed order $m\in \N$, and in this case we can take  $\alpha=2/m$ and $d(s,t)= c_m \| X(s)-X(t)\|_{L^2}$ for a suitable constant $c_m>0$.  
Processes satisfying \eqref{eq:673} with  $\a=1$ are  usually called sub-exponential.
From Theorem~\ref{thm-alpha} we derive the following result:

\begin{theorem}\label{cor-Int}
 Suppose that   $X=\{X(t)\!: t\in [0,1]\}$ 
 is a separable stochastic process satisfying  the increment condition \eqref{eq:673} for some $\a>0$.  Then the following \eqref{Int-item-1}--\eqref{Int-item-4} holds:
\begin{enumerate}
\item \label{Int-item-1} Suppose that  $\mathcal V(\Psi,d)<\infty$ for $\Psi(t)=t^p$,  $p>0$, and set  $\Phi(t)=t^{p} (\log^*_2(1/t) )^{-p/\alpha}$. Then  $X$ has sample paths of  bounded $\Phi$-variation almost surely.\smallskip
\item \label{Int-item-2}  Suppose that  $\mathcal V(\Psi,d)<\infty$ for $\Psi(t)=t^{p} (\log^*_2(1/t) )^{-p/\alpha}$ and $p>0$. 
Then $X$ has sample paths of bounded $p$-variation almost surely. \smallskip
\item \label{Int-item-3}     For all   $\b>0$ let $\Phi_{\b}(t)= 
e^{-t^{-1/\b}}$.
Suppose that  $\mathcal V(\Phi_{\b_0},d)<\infty$ for some $\b_0>1/\a$.  Then for all $\b<1-1/\a+\b_0$,  $X$ has  sample paths of bounded $\Phi_{\b}$-variation almost surely. \smallskip
\item \label{Int-item-4}   For all $c,r>0$ let $\Phi_{c,r}(t)=  e^{-r(\log \frac{1}{t})^{c}}$.  Suppose that $\mathcal V(\Phi_{c,r},d)<\infty$. Then for all $v>r$, $X$ has  sample paths of bounded $\Phi_{c,v}$-variation almost surely. 
\end{enumerate}
\end{theorem}

It follows from   Theorem~\ref{thm-fine-var}  below  that Theorem~\ref{cor-Int}\eqref{Int-item-1} gives the optimal result for all Hermite processes (including the  fractional Brownian motion and 
the Rosenblatt process).   In the special case where $X$ is a centered  Gaussian process we recover 
the above cited result by Jain and Monrad~\cite[Theorem~3.2]{JaMo} from Theorem~\ref{cor-Int}\eqref{Int-item-2} if we set $\a=2$. 
The functions $\Phi_{\b}$,  defined in \eqref{Int-item-3}, play an important role in Fourier theory, where it follows  from the Salem--Baernstein theorem   that 
every continuous periodic function of bounded $\Phi_{\b}$-variation has  uniform convergent  Fourier series  if and only if $\b>1$, see Remark~\ref{Four-con} for more details. In particular, all  continuous 
periodic functions of bounded $p$-variation for some $p\geq 1$ have an uniform convergent Fourier series. For $\a>1$ we have  $1-1/\a>0$ and hence we may choose $\b=\b_0$ in Theorem~\ref{cor-Int}\eqref{Int-item-3}, which  shows that $\mathcal V(\Phi_\b,d)<\infty$ implies that $X$ is of bounded $\Phi_\b$-variation. This should be compared with Proposition~\ref{on-ness-con}, which shows that when $X$ is centered Gaussian ($\a=2$) and the function $\Phi$ satisfies the $\Delta_2$-condition (i.e.\ $\Phi(2x)\leq C \Phi(x)$ for all $x$), then the opposite implication holds, that is, if  $X$ is of bounded $\Phi$-variation then $\mathcal V(\Phi,d)<\infty$. Notice, however,  that the functions $\Phi_\b$, considered in Theorem~\ref{cor-Int}\eqref{Int-item-3}, do not satisfy the $\Delta_2$-condition. In the special case where $c=1$ and $X$ is centered Gaussian,  
  Theorem~\ref{cor-Int}\eqref{Int-item-4} implies  the second statement of 
Jain and Monrad~\cite[Theorem~3.2]{JaMo}.

\medskip
 In all of the above cited papers by Taylor~\cite{Taylor}, Dudley and Norvai{\v{s}}a~\cite{Dud-Nor}, Kawada and K\^ono~\cite{Ka-Ko} and Marcus and Rosen~\cite{Mar-Ros, Mar-Ros-pa-1}    the limiting $\Phi$-variation of some Gaussian processes is moreover considered. To 
recall this notion we let, for any  $\delta>0$,   $\Pi_\delta$ denote  the set of all partitions $\pi=\{0=t_0<\dots<t_n=1\}$ of $[0,1]$ such that $|\pi |:=\max_{i=1,\dots,n} t_i-t_{i-1}\leq \delta$, and for $\pi\in \Pi_\delta$ set  
\begin{equation*}\label{v_PHi-12}
v_{\Phi}(f,\pi)=\sum_{i=1}^{n}
 \Phi\big(|f(t_{i})-f(t_{i-1})|\big).
\end{equation*}
The \emph{limiting $\Phi$-variation} of a function $f: [0,1]\to\R$ is  defined by 
\begin{equation}\label{def-lim-Phi}
\mathcal V^*_\Phi(f):=\lim_{\delta\to 0} \Big(\sup\{v_{\Phi}(f,\pi):\pi \in \Pi_\delta\}\Big).
\end{equation}
Let $X$ be a fractional Brownian motion with Hurst index $H\in (0,1)$, and $\Phi$ denote the function
 \begin{equation*}
  \quad \Phi(t)=\frac{t^{\frac{1}{H}}}{(\log^*_2(1/t))^{\frac{1}{2H}}}, \qquad t>0. 
 \end{equation*}
Dudley and Norvai{\v{s}}a~\cite[Theorem~12.12]{Dud-Nor} showed that with probability one
 \begin{equation}\label{fine-variation-23}
\lim_{\delta\to 0} \Big(\sup\{v_{\Phi}(X,\pi):\pi \in \Pi_\delta\}\Big)=2^{\frac{1}{2H}},
\end{equation}
which characterizes the limiting $\Phi$-variation of the fractional Brownian motion, and includes the Brownian motion case for $H=1/2$, which goes back to the fundamental work 
Taylor~\cite[Theorem~1]{Taylor}.  Eq.~\eqref{fine-variation-23}  implies that the fractional Brownian motion is of bounded $\Phi$-variation, and that $\Phi$ is optimal for in the sense described in Theorem~\ref{thm-fine-var} below. 
Kawada and K\^ono~\cite{Ka-Ko} and Marcus and Rosen~\cite{Mar-Ros, Mar-Ros-pa-1}   derived the limiting $\Phi$-variation of  other classes of Gaussian processes with stationary increments. 
 On the other hand, the limiting $\Phi$-variation of a symmetric $\alpha$-stable L\'evy process $X$ with $\alpha\in (0,2)$ is characterized by Fristedt and  Taylor~\cite[Theorem~2]{Fri-Tay} and from their result it follows that there does not exists an optimal $\Phi$-variation function for an $\alpha$-stable L\'evy process,  in contrast to the Brownian motion.   

\smallskip
In the following we will recall the definition of Hermite processes (see  Taqqu~\cite{Taqqu-2} or Tudor~\cite{Tudor}): 
Let $B$ be an independently scattered symmetric Gaussian random measure on $\R$ with Lebesgue intensity measure. An Hermite process of order $m\in \N$  with Hurst parameter $H\in (1/2,1)$ is a stochastic process  $X=\{X(t)\!: t\in [0,1]\}$ of the form 
  \begin{equation}\label{eq:74}
 X(t)=c_0 \int_{\R^m}  \Big(\int_0^t \prod_{i=1}^m (v-u_i)_+^{-(1/2+(1-H)/m)}\,\dd v\Big) \,\dd B(u_1)\cdots \dd B(u_m).
 \end{equation}
   Here $c_0>0$ is a positive norming constant and the integral in \eqref{eq:74} is a multiple Wiener--It\^o integral, see e.g.\  Nualart~\cite{N}. 
 Hermite processes appear  as the limit in non-central limit theorems, see Taqqu~\cite{Taqqu, Taqqu-2} or Tudor~\cite{Tudor}, they have stationary
increments, are   self-similar with index $H$, and for $m\geq 2$  are non-Gaussian. An Hermite process $X$ of order $m=2$ is usually called a  Rosenblatt 
process, and for $m=1$, $X$  is the fractional Brownian motion. We may and do assume that $X$ has been chosen with continuous sample paths almost surely. 

\vskip 2 pt Our next result is the  following theorem  which characterizes the limiting $\Phi$-variation of Hermite processes.

\begin{theorem}\label{thm-fine-var}
Let  $X=\{X(t)\!: t\in [0,1]\}$ be an Hermite process of order $m\in \N$  with Hurst index $H\in (1/2,1)$, and $\Phi=\Phi_{m,H}$  be given by  
 \begin{equation}\label{def-Phi}
\Phi(t)=\frac{t^{\frac{1}{H}}}{(\log^*_2( 1/t))^{\frac{m}{2H}}}\qquad t>0.  
\end{equation}
Then with probability one 
\begin{equation}\label{fine-variation}
\lim_{\delta\to 0} \Big(\sup\{v_{\Phi}(X,\pi):\pi \in \Pi_\delta\}\Big)=\sigma_{m,H}
\end{equation}
where the constant $\sigma_{m,H}\in (0,\infty)$ is defined in \eqref{def-sigma-X}.  In particular, we deduce that  $X$ has sample paths of bounded $\Phi$-variation with probability one, and $\Phi$ is optimal in 
the sense that  for all $\tilde \Phi:\R_+\to\R_+$ satisfying $\tilde \Phi(x) /\Phi(x)\to \infty$ as $x\to 0$, we have $\mathcal V_{\tilde \Phi}(X)=\infty$
a.s. Moreover,   $\sigma_{m,H}\leq \mathcal V_\Phi(X) <\infty$ a.s. 
\end{theorem}

We note that the  functions $\Phi=\Phi_{m,H}$ are decreasing in $m$, in particular, for all  $1\leq m< k$ we have  that 
$ \Phi_{k,H}(x)/\Phi_{m,H}(x)\to 0$ as $x\to 0$, and hence Theorem~\ref{thm-fine-var}
 shows that when $m$ increases then the sample paths of $X$ becomes more ``rough''  when measured in terms of $\Phi$-variation. In particular, the Rosenblatt
 process with Hurst index $H$ has more ``rough'' sample paths than the fractional Brownian motion.  
The next remark concerns the constant $\sigma_{m,H}$ appearing in Theorem~\ref{thm-fine-var}. 

\begin{remark}\label{remark-const}
The constant $\sigma_{m,H}$ in \eqref{fine-variation} is defined by 
\begin{equation}\label{def-sigma-X}
\sigma_{m,H}=2^{\frac{m}{2H}}\Big(\sup\Big\{\Big| \int_{\R^m} Q(u_1,\dots,u_m) \xi(u_1)\cdots\xi(u_m)\,\dd u_1\cdots \dd  u_m\Big|: \xi\in L^2(\R), \,\|\xi\|_2\leq 1\Big\}\Big)^{1/H}
\end{equation}
where 
\begin{equation*}
Q(u_1,\dots,u_m)=c_0  \int_0^1 \prod_{i=1}^m (v-u_i)_+^{-(1/2+(1-H)/m)}\,\dd v.
\end{equation*}
We note that $0<\sigma_{m,H}<\infty$, and in fact we have the upper bound  
 \begin{align}\label{eq:840}\notag
{}& \sigma_{m,H}^H\leq 2^{\frac{m}{2}}\|Q\|_{L^2(\R^m)} = 2^{\frac{m}{2}} (m!)^{-\frac{1}{2}} \| X(1)\|_{L^2} 
 \end{align}
and \begin{equation} \label{eq:841}  
\|Q\|_{L^2(\R^m)}=c_0 \Big[\frac{\beta(1/2-\frac{1-H}{m}, \frac{2-2H}{m})^m }{H(2H-1)} \Big]^{1/2}
 \end{equation}
  where $\beta(x,y):=\int_0^1 t^{x-1} (1-t)^{y-1}\,\dd t$ denotes the beta function evaluated in $(x,y)\in (0,\infty)^2$. Equality~\eqref{eq:841} follows by the proof of \cite[Proposition~3.1]{Tudor}.
By duality arguments it follows that $\sigma_{m,H}^H=2^{\frac{1}{2}}\|X(1)\|_{L^2}$ for  $m=1$,  whereas  $\sigma_{m,H}^H<2^{\frac{m}{2}} \|X(1)\|_{L^2}$ for $m\geq 2$. 
\end{remark}

\vskip 4 pt  The next   results are  related to Theorems~\ref{thm-alpha}--\ref{thm-fine-var}. We will in particular estimate the associated $\Phi$-variation norm of $X$. 
Some notions and notation are necessary.  Let $\mathcal
B_\Phi$  be the class      of all  real functions $f:[0,1]\to \R$ 
 such that $\mathcal   V_\Phi(f)<\infty$,    namely having  bounded $\Phi$-variation. Recall  some basic facts in the case where $\Phi$ is convex. Then   $\Phi$ is also continuous.  The class $\mathcal B_\Phi$   is a symmetric and convex set, but is not necessarily a linear space.  
Musielak and Orlicz \cite{MuO} haved proved that $\mathcal B_{\Phi} $ is a vector space if and only if $\Phi$ verifies the
$\D_2$-condition ($\Phi(2x)\le C\Phi(x)$ for all   $x$).  
In the latter case we define for any $f\in \mathcal B_{\Phi}$,
\begin{equation*}
\| f\|_\Phi = \inf\{r>0 : \mathcal V_\Phi(f/r)\leq 1\},
\end{equation*}
  and  $\lvert \lvert\lvert f\rvert \rvert \rvert_\Phi = |f(0)|+\| f \|_\Phi$.  We know by a theorem of Maligranda and Orlicz \cite{MaO},  that the space $( \mathcal B_\Phi ,\lvert \lvert\lvert \cdot\lvert \lvert\lvert_{\Phi})$ is a Banach space, and in fact a Banach algebra. 
  When $\Phi(x)= |x|^p$, $p\geq 1$,  we denote this space $\mathcal B_p$, and we will write $\mathcal V_\Phi(\cdot)= \mathcal V_p(\cdot)$ and $\|\cdot \|_\Phi=\| \cdot \|_p$; we note that the latter takes the simple form $\| f\|_p = \mathcal V_p(f)^{1/p}$. In the following we give explicit estimates on the expected $\Phi$-variation norm of Gaussian processes under some additional assumptions on $\Phi$ and $\Psi$. In particular, Corollary~\ref{cor-int} completes Jain and Monrad's  result \cite[Theorem~3.2]{JaMo} with an explicit estimate of the expected $p$-variation norm. We note that in Jain and Monrad's  result, $\Phi$ and $\Psi$ are given 
  and condition $\mathcal V(\Psi, d)<\infty$  
 reads on $d$. Here we adopt the same point of view and we notice   that our main result can also be read this way for Gaussian processes, namely   with $\Phi$ and $\Psi$ satisfying \eqref{con-1-thm1-alpha}, next $X$ subject to satisfy   condition $\mathcal V(\Psi, d)<\infty$ (the   increment condition \eqref{eq:673} is automatically  satisfied with $\a=2$). 

\begin{theorem} \label{t1} 
Suppose that $\Phi,\Psi:\R_+\to\R_+$ are strictly increasing continuous functions with $\Phi(0)=\Psi(0)=0$, $\lim_{t\to\infty} \Phi(t)=\lim_{t\to\infty} \Psi(t)=\infty$ such that $\Phi$ and $\Psi$ satisfy \eqref{con-1-thm1-alpha} of 
Theorem~\ref{thm-alpha}. In addition, suppose that $\Phi$ is convex and $\Psi^{-1}$ is absolute continuous with a derivative $(\Psi^{-1})'$ satisfying  
\begin{equation}\label{eq:436728}
 (\Psi^{-1})'( x y)\leq   K_0 x^{q-1} (\Psi^{-1})'(y),  \qquad \ \text{and that} \qquad \   \Phi(x y) \leq   K_0 x^{p} \Phi(y)
\end{equation}
for all $x,y\geq 0$ and some constants $p,q,K_0>0$. 
For all separable, centered Gaussian processes  $X=\{X(t)\!: t\in [0,1]\}$ with 
$d(s,t)=\|X(s)-X(t)\|_{L^2}$ for all $s,t\in [0,1]$,  such that
$\mathcal V(\Psi,d)<\infty$, 
the following estimate holds 
 \begin{equation*}
\E \| X \|_\Phi \leq  K \Big(\mathcal V(\Psi,d)+\mathcal V(\Psi,d)^{p/2+1}+\mathcal V(\Psi,d)^{p q}\Big)^{1/p} 
\end{equation*}
where $K$ is a finite constant not depending on  process $X$. 
%
%
   \end{theorem}

   \begin{corollary}\label{cor-int}
Let    $p\geq 1$  and set 
 \begin{equation*}\label{eq:J-M-cond}
\Psi(x) = x^p[\log_2^*(1/(x\wedge 1))]^{p/2}\qquad x>0.
 \end{equation*}
 For all  separable, centered Gaussian processes  $X=\{X(t)\!: t\in [0,1]\}$ with  $d(s,t)=\|X(s)-X(t)\|_{L^2}$ for all $s,t\in [0,1]$, the estimate holds 
  \begin{equation}\label{int-var-X}
\E\| X \|_p \leq  K \mathcal W_p^{1/p} (1+\mathcal W_p^{1/2})\qquad \text{where}  \qquad \mathcal W_p:=\mathcal V(\Psi,d)
 \end{equation}
 and $K$ is a finte constant  not depending on process $X$. 
\end{corollary}
 Recall that in the setting of Corollary~\ref{cor-int} the condition $\mathcal W_p<\infty$ is Jain and Monrad's sufficient conditions of bounded $p$-variation. 
In the case where $\mathcal W_p\leq 1$, \eqref{int-var-X} simplifies to the estimate $\E\|X \|_p\leq K  \mathcal W_p^{1/p}$.   The next result gives a necessary condition for a centered Gaussian process to  be of bounded $\Phi$-variation in the case where $\Phi$ satisfies the $\Delta_2$-condition, i.e.,  there exists a finite constant $C$ such that $\Phi(2x)\leq C \Phi(x)$ for all $x\geq 0$.

   \begin{proposition}\label{on-ness-con}
  Suppose that $X=\{X(t)\!: t\in [0,1]\}$ is a separable, centered Gaussian   process,  $d(s,t)=\|X(s)-X(t)\|_{L^2}$ for $s,t\in [0,1]$,   $\Phi$ is convex and satisfies the $\Delta_2$-condition.    If $X$ is of bounded
  $\Phi$-variation then $\mathcal V(\Phi,d)<\infty$. 
     \end{proposition}

Next  we discuss some related results.  For  $p$-variation of 
Markov processes we refer to  Man\-stavi{\v{c}}ius~\cite{Man-Markov}, for $p$-variation of stable processes we refer to Xu~\cite{Xu}, and for results on the $p$-variation of the local time of  stable L\'evy processes in the
space variable we refer to Marcus and Rosen~\cite{Mar-Ros, Mar-Ros-pa-1}. A result of  Vervaat \cite{V}, states that if $X$ is a self-similar process with index $0<H\le 1$ and has moreover stationary increments, then its sample paths  have nowhere bounded variation,  unless $X(t)\equiv tX(0)$.  The weak variation of a class of  Gaussian processes is characterized in Xiao~\cite{Xi}, and 
  various results on the   moduli of continuity for  Gaussian random fields are obtained in  Meerschaert et al.~\cite{Meer}.    In a different direction,  the asymptotics of the renormalized quadratic variation of the Rosenblatt process  is analyzed in Tudor and Viens~\cite{TV} using Malliavin calculus.  
  
      \vskip 3pt 
  
The paper is organized as follows:     Theorems~\ref{thm-alpha}--\ref{cor-Int}  are  proved in Section~\ref{sec4}. Theorem  \ref{thm-fine-var} is proved in Section~\ref{sec3}.   We prove Theorem~\ref{t1}, Corollary~\ref{cor-int} and    Proposition~\ref{on-ness-con} in Section \ref{section4}.   Finally, some results used in the proofs are moved to the Appendix.

 \section{Proofs of Theorems~\ref{thm-alpha} and \ref{cor-Int}.}  \label{sec4}
 
 \subsection{Preliminaries.} \label{sec-pre}
Recall some basic facts. 
 Let $d$ be a pseudo-metric on $[0,1]$   with finite diameter $D$.
Moreover, let  $X=\{X(t)\!: t\in [0,1]\}$ be a    stochastic process with basic (complete) probability space $(\O,\A,\P)$. 
   We further say  that $X$ is  separable (for the usual distance on $[0,1]$), if  there exists a countable subset $\mathcal S$ of $[0,1]$ (separation set) and a null set $N$ of $\A$, such that for any $\o\in N^c$ and any $t\in [0,1]$, there is a sequence $\{t_n, n\ge 1\}\subset S$ verifying $t_n\to t$ and $X(t,\o)=\lim_{n\to \infty } X(t_n,\o)$.  
By \cite[Ch.~4.2, Theorem~1]{Gik-Sko}, every stochastic process  $X= \{X(t)\!: t\in [0,1]\}$  admits a separable version $\widetilde X$ with values in the extended reals.
%
 The measurability of the functional $\mathcal  V_\Phi(X)$ follows from the separability assumption on $X$.
For each continuous, strictly increasing function $f:\R_+\to \R$, $f^{-1}$ will  denote its  inverse. Recall  that  $\log^*(x)=\log (x+1)$, $\log_2^*(x)=\log^*(\log^* x)$ for all $x\geq 0$, and that   $\Phi: \R_+\to \R_+$ denotes   a continuous, strictly  increasing    function with  $\Phi(0)=0$ and $\lim_{t\to \infty} \Phi(t)=\infty$.

\vskip 2 pt

 The following  metric entropy result, Theorem~\ref{t2},   plays a key role in the proof of Theorem~\ref{thm-alpha} both to obtain good modulus of continuity estimates on $X$, but also to estimate certain probabilities.  It relies on the  covering numbers $N(T,d,\e)$, $\e>0$,  of a  pseudo-metric space $(T,d)$,  which are defined to be the  smallest number of open balls of radius $\e$ to cover $T$.      Note also that $\phi_\a^{-1}(x)=(\log^* x)^{1/\alpha}$,  $x\geq 0$. 
 
\begin{theorem}\label{t2}    Let $(T,d)$ be a pseudo-metric space  with diameter $D$,   and let $X=
\{X(t)\!: t\in T\}$ be a separable stochastic process   such that the increment condition \eqref{eq:673} is fulfilled for some $0<\alpha<\infty$ and with $[0,1]$ replaced by $T$. 
Further assume that  
\begin{eqnarray*}
\int_0^D   ( \log^* N(T,d,\e ))^{1/\a} \, \dd \e <\infty.  
\end{eqnarray*} 
 Let  $\d(\e)=  \int_0^{\e } (\log^* N(T,d,u))^{1/\a} \, \dd u$, $\e\in [0, D]$.   Then  there exists a finite constant $K_\a$ depending on $\a$ only such that
 \begin{eqnarray*} \Big\|\sup_{s,t\in T} \frac{|X(s)-X(t)|}{\d(d(s,t))}\Big\|_{\phi_\a} 
  &\le &   K_\a.\end{eqnarray*}  
 \end{theorem}

 \begin{remarks} \label{r2}   (i): It is possible to deduce Theorem~\ref{t2} from 
 Kwapie\'n and Rosi\'nski~\cite{Kwa-Ros}, Corollary~2.2 and Remark~1.3(2), by  modifying the functions $\phi_\a$ in a suitable way in the case $\a<1$  to obtain a Young function.  The result of \cite{Kwa-Ros} relies on the  majorizing measure method, and hence does also apply under the metric entropy integral condition  by a well known existence result.  For the reader's convenience we give a direct proof of Theorem~\ref{t2} using  metric entropy methods  and which holds for any   $\a>0$ in Appendix~\ref{appendix-1}.

 \vskip 1 pt  \noi  (ii): Theorem \ref{t2} 
   implies   that 
  \begin{eqnarray*} 
 \big\|\sup_{s,t\in T} {|X(s)-X(t)|} \, \big\|_{\phi_\a} 
  &\le &  K_\a \int_0^{D }(\log^*  N(T,d,\e))^{1/\a}\,  \dd \e, 
  \end{eqnarray*}  
a   classical result. And at this regard,  it  is a natural complement of this one. 
\end{remarks} 

\subsection{Proof of Theorem~\ref{thm-alpha}}  \label{sec4-1}
  Put  for $0\le t\le 1$,
  \begin{equation*}\label{def-of-F}
F(t)=\sup_{0=t_0<\dots<t_n=t\atop n\in \N} \sum_{i=1}^n  \Psi\big( d(t_i,t_{i-1})\big).
\end{equation*}
The basic observation is that if   $  s\le t $, then $ \Psi\big( d(s,t )\big)\le F(t) -F(s)$.  
Hence, 
 \begin{equation}\label{base1}
d(t,s)\le     \Psi^{-1}\big(F(t)-F(s)\big) \qquad\quad 0\le s\le t\le 1,\end{equation}
and in particular $D\le \Psi^{-1}(F(1)) $ (notice that $F(1)= \mathcal V(\Psi, d)$). Introduce for all $m\in \Z$ and 
$j=1,\ldots, \lceil F(1) 2^m\rceil$, the sets 
  $$I_{m,j}=](j-1)2^{-m}, j2^{-m}], \qq\quad S_{m,j} =\{s\in [0,1]:   F(s)\in I_{m,j} \} .$$
 Put  also 
 \begin{equation*}
 Z_m      =   \sharp\big\{j=1,\dots,\lceil F(1) 2^m\rceil: \sup_{s,t\in S_{m,j}} \Phi(|X(t )-X(s )|)>C 2^{-m}\big\},
 \end{equation*}
where $\sharp A$ denotes the number of elements in a set $A$, and $\lceil x\rceil$ denotes the least integer larger than  $x$. 
 We have $d(s,t)\le \Psi^{-1}(2^{-m})$ for all $s,t$ such that $F(s), F(t)\in I_{m,j}$. Let $0<\rho<   2^{-m}/2$.  
Consider a subdivision of $I_{m,j}$   of size $ 2\rho$, thus of 
length $N= N(\rho)$  bounded by
$ \big[  2^{-m}/\rho\big] +1 \le  2^{-m+1}/\rho$. It induces a partition $U_1, \ldots, U_{N(\rho)}$ of $S_{m,j}$. 
 For any non-empty element $U$ from this partition,   any $s,t\in U$, we have $\Psi^{-1}(|s-t|)\le \Psi^{-1}(2\rho)$. Pick $u$ arbitrarily  in  $U$, we see that the $d$-ball centered at $u$ and  of radius  $ \Psi^{-1}(2\rho) $, contains $U$. Moreover, repeating this operation for
each non-empty element $U_1, \ldots, U_{N(\rho)}$, we manufacture like this a  covering of $S_{m,j}$   by   
$d$-balls of radius  $ \Psi^{-1}(2\rho) $, centered in $S_{m,j}$ of size at most $N $. Thus    
 $N( S_{m,j},d, \Psi^{-1}(2\rho)) \le   2^{-m+1}/\rho $. 
     Letting $\rho = \Psi(\e)/2$, we deduce 
\begin{equation}\label{entro}
N( S_{m,j},d, \e)  \le   \frac{ 2^{-m+2}}{\Psi(\e)}  , \qq \quad  0<\e <\Psi^{-1}(2^{-m}) .
\end{equation}  
 In particular, for $m=0$ and by considering $\tilde F(t) = F(t)/F(1)$ and $\tilde \Psi(t)=\Psi(t)/F(1)$ instead of $F$ and $\Psi$ we obtain that 
  \begin{equation}\label{entro-2}
N( [0,1],d, \e)  \le   \frac{ 4 F(1) }{\Psi(\e)}  , \qq \quad 0<\e<\Psi^{-1}(F(1)).
\end{equation}

By Remark~\ref {r2}(ii) and \eqref{entro},   
\begin{eqnarray*} \big\|\sup_{s,t\in S_{m,j}} {|X(s)-X(t)|}\, \big\|_{\phi_\a} 
 &\le &   K_\a \int_0^{\Psi^{-1}(2^{-m}) }\Big(\log^* \Big(\frac{2^{-m+2}}{\Psi(\e)}\Big) \Big)^{1/\a} \dd \e\leq K_\alpha 4^{1/\a} y_m 
 \end{eqnarray*}  
 where we use the estimate 
 \begin{equation}\label{est-log}
  \log^*( y x)\leq (y\vee 1) \log^*(x)\qquad \text{for all }y,x\geq 0.
 \end{equation}
With $C$ given as in the theorem and  $x_m$ defined in \eqref{con-1-thm1-alpha} we have
   \begin{eqnarray*}
& &\P\Big\{\sup_{s,t\in S_{m,j}}\Phi( |X(s)-X(t)|)> C2^{-m}\Big\} 
 \cr  
&\le & \P\Big\{\sup_{s,t\in S_{m,j}}  |X(s)-X(t)| > x_m\, \big\|\sup_{s,t\in S_{m,j}} {|X(s)-X(t)|} \,\big\|_{\phi_\a}\Big\}
\leq  \frac{1}{\phi_\a(x_m)}. \label{lsfjsfl}
\end{eqnarray*}
  And this holds for $ j=1,\ldots,  \lceil F(1) 2^m\rceil$.  By letting  $r_0=\phi_{\alpha}^{-1}(1)$ and $\tilde x_m=x_m\vee r_0$  we have  
  \begin{equation}\label{eq:942}
  \P\Big\{\sup_{s,t\in S_{m,j}}\Phi( |X(s)-X(t)|)> C2^{-m}\Big\} \leq \frac{1}{\phi_\a(\tilde x_m)} 
  \end{equation}
  since the left-hand side of \eqref{eq:942} is always  less than or equal to one. 
  
 \vskip 1pt
 Fix now a  partition $\pi=\{0=t_0<\dots<t_n=1\}$ of $[0,1]$.
Let $m_0\in\Z$ be any number satisfying $\max_{1\leq i\leq n} F(t_i)-F(t_{i-1})\leq 2^{-m_0}$. 
 Introduce the following sets
\begin{eqnarray}
\Lambda_m &= &\Big\{  1\le i\le n:  2^{-m-1} <F(t_i)-F(t_{i-1})\leq 2^{-m}\Big\},\qquad m\geq m_0,
\cr E(\omega)&=&\Big\{  1\le i\le n:  \Phi\big( |X(t_i,\omega)-X(t_{i-1},\omega)|\big)>  C [F(t_i)-F(t_{i-1})]\Big\}. \label{def-of-E-omega}
\end{eqnarray}
 We have
  \begin{align}\notag
{}& \sum_{i=1}^{n} 
\Phi\big( |X(t_i,\omega)-X(t_{i-1},\omega)|\big)  =\sum_{i\in E(\omega)} 
\Phi\big( |X(t_i,\omega)-X(t_{i-1},\omega)|\big)
\\ \label{dep} {}& \quad +\sum_{i\notin E(\omega)} \Phi\big( |X(t_i,\omega)-X(t_{i-1},\omega)|\big) \le  \sum_{i\in E(\omega)} \Phi
\big( |X(t_i,\omega)-X(t_{i-1},\omega)|\big)+ C F(1).  
\end{align}
 In order  to control the remainding subsum, 
 we start with the following simple bound, 
 \begin{align}\notag
{}&  \sum_{i\in E(\omega)} \Phi\big( |X(t_i,\omega)-X(t_{i-1},\omega)|\big)=\sum_{m=m_0}^\infty \Big(\sum_{i\in E(\omega)\cap \Lambda_m}
\Phi\big( |X(t_i,\omega)-X(t_{i-1},\omega)|\big) \Big)
 \\ \label{est-E-2} {}& \qquad \le  \sum_{m=m_0}^\infty\Big(\sup_{i\in E(\omega)\cap \Lambda_m}\Phi\big( |X(t_i,\omega)-X(t_{i-1},\omega)|\big)\Big)\sharp \big(E(\omega)\cap
\Lambda_m\big).
 \end{align}

We now need some auxiliary estimates. 
We first claim that 
\begin{equation}\label{sdlfjsa}
\sharp \big(E(\omega)\cap \Lambda_m\big)\leq 9 Z_{m-2}(\omega)\qquad \text{for all }m\geq m_0.
\end{equation}
  For  $i\in E(\omega)\cap \Lambda_m$ choose an  integer  $j=1,\dots, \ldots, \lceil F(1) 2^m\rceil$  such that $F(t_i)\in I_{m,j}$.   Then $F(t_{i-1})\in I_{m,j-1}\cup I_{m,j}=J$, say.
For some   $k=1,\dots, \lceil F(1) 2^{m-2}\rceil$, we have $J\subseteq I_{m-2,k}$ 
which by definition of $\Lambda_m$ and $E(\omega)$ implies that 
\begin{equation}\label{eq2}
\sup_{s,t\in S_{m-2,k}} \Phi(|X(t,\omega)-X(s,\omega)|)>C (F(t_i)-F(t_{i-1}))\geq   C 2^{-(m-2)}.
\end{equation}
On the other hand, for each $k=1,\dots, \lceil F(1) 2^{m-2}\rceil$ we have 
\begin{equation}\label{eq1}
\sharp\{i\in E(\omega)\cap \Lambda_m:F(t_i)\in I_{m-2,k}\}\leq 9.
\end{equation}
 Indeed, to show \eqref{eq1} let $i_1$ and $i_0$ be the largest and least integers in  $A_k:=\{i\in E(\omega)\cap 
\Lambda_m:F(t_i)\in I_{m-2,k}\}$. Since $F(t_{i_0}),F(t_{i_1})\in I_{m-2,k}$, by definition, we have   
 \begin{eqnarray*}
 2^{-(m-2)}&\ge& F(t_{i_1})-F(t_{i_0}) =  \sum_{i\in A_k:\, i>i_0} [F(t_i)-F(t_{i-1})]\cr &\ge&   \sum_{i\in A_k:\, i>i_0} 2^{-m-1}\geq  2^{-m-1}\big(\sharp A_k-1\big),
 \end{eqnarray*}
which shows \eqref{eq1}. 
 Combining \eqref{eq2} and \eqref{eq1} we obtain \eqref{sdlfjsa}.  
By \eqref{est-E-2} and \eqref{sdlfjsa},   
 \begin{eqnarray*}
     \sum_{i\in E(\omega)} \Phi\big( |X(t_i,\omega)-X(t_{i-1},\omega)|\big)  
     & \le &        9\sum_{m= m_0}^\infty  Z_{m-2}(\omega)
  \sup_{i\in E(\omega)\cap \Lambda_m}\Phi\big( |X(t_i,\omega)-X(t_{i-1},\omega)|\big).
  \end{eqnarray*}  
 
  Besides, by Theorem~\ref{t2}   there exists a finite constant $K_\a>0$ depending on
   $\alpha$ only such that if 
 \begin{eqnarray}\label{theta}\Theta= \sup_{s,t\in [0,1]} \frac{|X(s)-X(t)|}{\d(d(s,t))} ,
   \end{eqnarray} then 
$\|\Theta\|_{\phi_\a}\le K_\a$ ($\delta$ is defined in Theorem~\ref{t2} with $T=[0,1]$).  
 Now, if $i\in E(\omega)\cap \Lambda_m$, we have by \eqref{base1},   $   d(t_i, t_{i-1} ) \le \Psi^{-1}\big(F(t_{i }) -F(t_{i-1})\big)\le
\Psi^{-1}( 2^{-m})$,
 and  hence 
  \begin{eqnarray*}\Phi\big( |X(t_i,\omega)-X(t_{i-1},\omega)|\big)&\le & 
\Phi\big( \Theta(\omega)\d\big(d(t_i, t_{i-1} )\big)\big)\le \Phi\big( \Theta(\omega) \d\circ \Psi^{-1}( 2^{-m})\big).   
\end{eqnarray*}  
By \eqref{entro-2} and \eqref{est-log} we have with $c_d= [(4F(1))\vee 1]^{1/\a}$ that 
\begin{align*}
\delta\circ \Psi(2^{-m})\leq  {}& \int_0^{\Psi(2^{-m})} \Big(\log^*\Big(\frac{ 4 F(1) } {\Psi(\e)}\Big)\Big)^{1/\a} \,\dd \e 
\\ \leq {}& c_d  \int_0^{\Psi(2^{-m})} \Big(\log^*\Big(\frac{ 1 } {\Psi(\e)}\Big)\Big)^{1/\a} \,\dd \e = c_d y_m, 
\end{align*}
and hence  
\begin{equation}  \label{sdhkfs}
    \sum_{i\in E(\omega)} \Phi\big( |X(t_i,\omega)-X(t_{i-1},\omega)|\big)   
      \le     9 \sum_{m=m_0}^\infty  \Phi\big( \Theta(\omega) c_d   y_m\big)
  Z_{m-2}(\omega).
  \end{equation}
In the following we will show that the right-hand side of \eqref{sdhkfs} is finite almost surely.
 For all $\epsilon>0$ choose $M$ such that $\P(\Theta\leq M)\geq 1-\epsilon$. For  $\o\in \{\Theta\leq M\}$, 
  \begin{equation*}
   \sum_{m=m_0}^\infty  \Phi\big( \Theta(\omega) c_d y_m\big)
  Z_{m-2}(\omega)  \le \sum_{m=m_0}^\infty  \Phi\big( M c_d y_m \big)
  Z_{m-2}(\omega) =:    \mathcal Z_M(\o). 
  \end{equation*}
Suppose that  $m_0$ is the greatest  integer such that $F(1)\leq 2^{-m_0}$, 
  and hence for all $m\geq m_0$   we have that $F(1)2^m\geq 1/2$, which implies that $\lceil F(1)2^m\rceil\leq 2 F(1) 2^m$.  
    By definition of $Z_m$ and \eqref{eq:942} we have for all $m\geq m_0$, 
  \begin{eqnarray}\label{zm1}
  \E Z_m&= & \sum_{j=1}^{\lceil F(1) 2^m\rceil }\P\Big\{\sup_{s,t\in S_{m,j}}\Phi( |X(s)-X(t)|)> C 2^{-m}\Big\} \cr &\leq &  \lceil F(1)
2^m\rceil \frac{1}{\phi(\tilde x_m)}\leq 2 F(1) 2^m \frac{1}{\phi_\a(\tilde x_m) }\leq 4 F(1) 2^m e^{-x_m^\alpha}
  \end{eqnarray}
  where we have used that $1/\phi_\a(\tilde x_m)\leq 2 e^{-x_m^{\alpha}}$ in the last inequality. Hence
  \begin{equation*}\label{zm2}
    \E \mathcal Z_M  =  \sum_{m=m_0}^\infty  \Phi\big( M c_d  y_m\big)
 \E   Z_{m-2}  
\le   4F(1) \sum_{m=m_0}^\infty    2^{m}\Phi\big( M  c_d y_m \big)   e^{-x_m^\a}<\infty
  \end{equation*}
  by assumption \eqref{con-1-thm1-alpha}. 
  That is,
 the random variable in  \eqref{sdhkfs} is finite with probability greater than $1-\epsilon$ for all $\epsilon>0$  and hence finite
almost surely.  By \eqref{dep}
\begin{equation}  \label{eq:2424}
\sum_{i=1}^n \Phi\big( |X(t_i,\omega)-X(t_{i-1},\omega)|\big)     \le   9 \sum_{m=m_0}^\infty  \Phi\big( \Theta(\omega) c_d y_m \big) Z_{m-2}(\omega)   +C F(1).   
\end{equation} 
This inequality holds for all $\o$ in a measurable set $\O_0$ of probability 1 and all partitions
$\pi$.   Note also that the right-hand side   {\it does not} depend
on the given partition
$\pi$; it only depends on the process $X$. So that   we can take  supremum over all partitions
$\pi$ on  the left-hand side. And recalling in view of the separability assumption that $\mathcal   V_\Phi(X  )$ is measurable, we have shown 
that for   all
$\o\in \O_0$,
\begin{equation*} \label{dep2-2}
  \mathcal   V_\Phi(X(\cdot,\omega))     \le    9 
  \sum_{m=m_0}^\infty  \Phi\big( \Theta(\omega) c_d y_m \big) Z_{m-2}(\omega)
+C F(1), 
 \end{equation*} 
which  shows  that $X$ has sample paths of bounded $\Phi$-variation almost surely, and completes the proof. \qed

  \begin{remark}\label{entropy} 
  Suppose that $\mathcal V(\Psi,d)<\infty$ with $\Psi(x) = x^p$, and 
 let  $\d(\e)$ be defined as in Theorem~\ref{t2} with $T= [0,1]$.  
  Then the   estimate \eqref{entro-2} shows that 
$$\d(\e)\le   \int_0^{\e } \big(\log^* \frac{4 }{\e^p}\big)^{1/\a}\,  \dd u\le C_{p,\a} \ \e \big(\log^* \frac{1 }{\e }\big)^{1/\a} .$$
Specify this for $X$ Gaussian  ($\a=2$). We deduce from Theorem \ref{t2}, that 
 \begin{eqnarray*} \Big\|\sup_{s,t\in T} \frac{|X(s)-X(t)|}{ d(s,t) (\log^* \frac{1 }{d(s,t) })^{1/2}}\Big\|_{\phi_2} 
  &\le &   K ,
  \end{eqnarray*}  
 which justifies the implication \!\eqref{eq:discussJaMo} $\Rightarrow$ \eqref{eq23} invoked in the Introduction. 
Furthermore, the above estimates remain valid under the conditions considered in Theorem~\ref{cor-Int}\eqref{Int-item-2}.  
 \end{remark}
 
\begin{remark} The following example provides a good illustration of Jain and Monrad result (see Introduction). Taking  
$d(s,t)=d_1(|s-t|)$ where
$d_1(u)\leq  u^{1/p}\log^*_2(1/u)^{-1/2}$ and $\Psi(t)=t^p(\log_2^*(1/t))^{p/2}$ for  $t>0$, we see that  $\mathcal V(\Psi,d)<\infty$, so that $X$ has sample paths of bounded $p$-variation almost surely.  Further, from Remark~\ref{entropy} 
$$| X(t,\omega)-X(s,\omega)| \le  C(\omega)  \frac{|s-t|^{1/p}\ (\log^* \frac{1}{|s-t|} )^{1/2}}{ (\log^*_2 \frac{1}{|s-t|})^{ 1/2}}    \qquad \forall 0\le s,t\le 1.
$$ 
thereby implying that  has sample paths of bounded $\Psi_1$-variation almost surely, 
where $\Psi_1(x)= |s-t|^{p}  (\log^*_2 \frac{1}{|s-t|} )^{-p/2} ( \log^*\frac{1}{|s-t|})^{ p/2} $, which is clearly weaker than bounded $p$-variation.  Thus, we can not obtain bounded  $p$-variation  only from the modulus of continuity of $X$. \end{remark}

%
%
%

Next we will show  Theorem~\ref{cor-Int} by suitable applications of   Theorem~\ref{thm-alpha}.

\subsection{Proof  of Theorem~\ref{cor-Int}}

Cases \eqref{Int-item-1}--\eqref{Int-item-4} follow from Theorem~\ref{thm-alpha}, once we have shown that the sum  \eqref{con-1-thm1-alpha} is finite. The
two sequences $(x_m)_{m\geq 1}$ and $(y_m)_{m\geq 1}$ are defined  in \eqref{con-1-thm1-alpha-2}. In the following $K$ with subscript will denote  a finite
constant  only depending on the subscript, but might vary throughout the proof.  Given two sequences $(a_m)_{m\geq 1}$ and $(b_m)_{m\geq 1}$,  we  write   $a_m \sim b_m$ as
$m\to \infty$ if $a_m/b_m\to 1$ as $m\to \infty$, and similar for functions.   
  
\vskip 2 pt
\noindent
 \emph{Case~\eqref{Int-item-1}}: Set  $\Psi(t)=t^p$ and $\Phi(t)=t^p(\log_2^*(1/t))^{-p/\alpha}$ for all $t>0$ and $\Phi(0)=\Psi(0)=0$.  Since      $\Psi^{-1}(t)=t^{1/p}$ we have that
 \begin{equation*}
 y_m= \int_0^{2^{-m}} \Big(\log^*\Big(\frac{1}{\e}\Big)\Big)^{1/\a}\,\Psi^{-1}(\dd \e) = 
(1/p)\int_0^{2^{-m}} \Big(\log^*\Big(\frac{1}{\e}\Big)\Big)^{1/\a} \e^{1/p-1}\,\dd \e \sim (1/p) m^{1/\a} 2^{-m/p}
 \end{equation*}
 as $m\to\infty$ since  the integrand is regularly varying. 
 For all $m\geq 1$, we have by substitution that 
 \begin{equation*}
 \int_0^{2^{-m}} \Big(\log^*\Big(\frac{2^{-m}}{\e}\Big)\Big)^{1/\a} \,\Psi^{-1}(\dd \e)= K_{\a,p}  2^{-m/p}. 
 \end{equation*}
We note that   $\Phi^{-1}(t)\sim t^{1/p} (\log_2^*(1/t))^{1/\a}$ as $t\to 0$, and hence, for all  $C>0$  and all $m$ larger than some $m_0\geq 1$
\begin{equation*}
 x_m=\frac{\Phi^{-1}(C 2^{-m})}{K_\a\int_0^{2^{-m}} (\log^*(\frac{2^{-m}}{\e}))^{1/\a} \,\Psi^{-1}(\dd \e)}\geq  
 \frac{ K_{\a,p}C^{1/p} 2^{-m/p}\log(m)^{1/\a}}{2^{-m/p}}= K_{\a,p} C^{1/p}\log(m)^{1/\a}. 
 \end{equation*}
For all $M>0$,
 \begin{align*}
   \sum_{m=m_0}^\infty 2^m \Phi(M y_m) e^{-x_m^\a}&\leq 
 M^{p} K_{\a,p} \sum_{m=m_0}^\infty 2^m 
  2^{-m } m^{p/\a}   e^{- K_{\a,p}^\alpha C^{\a/p} \log m}  \\ 
&   =  
  M^{p}K_{\a,p}    \sum_{m=m_0}^\infty  m^{p/\a-K_{\a,p}^\alpha C^{\a/p}}  <\infty
 \end{align*}
 for  $C= (p/\alpha +2)^{p/\alpha} K_{\alpha, p}^{-p}$. This shows \eqref{con-1-thm1-alpha} and completes the proof of \eqref{Int-item-1}.
 
  \vskip 3 pt
 \noindent\emph{Case~\eqref{Int-item-2}}: The proof of  this case  is  analogous with 
the one of   case \eqref{Int-item-1}. We set $\Phi(t)=t^p$ and $\Psi(t)=  {t^{ p}}{(\log^*_2(1/t))^{p/\alpha}}  $.   As $t\to 0$ we have $\Psi^{-1}(t)\sim t^{1/p}/ (\log_2^*(1/t))^{ 1/\a}$ and  $(\Psi^{-1})'(t)\sim (1/p) t^{1/p-1} /(\log_2^*(1/t))^{1/\a}$. 
Hence 
\begin{align*}
y_m = \int_0^{2^{-m}} \log^*\Big(\frac{1}{\e}\Big)^{1/\a}\,(\Psi^{-1})'(\e)\,\dd \e\leq K_{\a,p} 2^{-m/p} m^{1/p},
\end{align*}
and therefore
$$ \Phi(My_m)\le M^p K_{\a, p}  2^{-m}m^{p/\a}   $$ for all $m$ large enough.   
Further, $x_m\geq K_{\a, p}  C^{1/p} (\log m)^{1/\a}$ for all $m$ large enough. Indeed,   
 \begin{eqnarray*} 
 \int_0^{2^{-m}} \Big(\log^*\Big( \frac{2^{-m}}{\e}\Big)\Big)^{1/\a} \,\Psi^{-1}(\dd \e)&=&   2^{-m}\int_0^{  1 }\Big( \log^*
 \frac{1}{v}\Big)^{1/\a}(\Psi^{-1})'(2^{-m}v) \,\dd v\cr &\leq & K_{\a, p}  2^{-m}\int_0^{  1 }\Big( \log
 \frac{1}{v}\Big)^{1/\a} (2^{-m}v)^{\frac{1}{p}-1}\frac{\dd v}{(\log_2^*\frac{2^m}{v})^{1/\a}}\cr &\le  &K_{\a, p} 
 \frac{2^{-m/p}}{(\log m )^{1/\a}}\int_0^{  1 }\big( \log
 \frac{1}{v}\big)^{1/\a}  v^{\frac{1}{p}-1}\,\dd v \cr &\le  &K_{\a, p} 
 \frac{2^{-m/p}}{(\log m )^{1/\a}}.
 \end{eqnarray*}
  Thus
$$x_m=\frac{ \Phi ^{-1}( C 2^{-m})}{K_\a\int_0^{2^{-m}}\big( \log^*( \frac{ 2^{-m}}{\e})\big)^{1/\a} \,\Psi^{-1}(\dd \e)}
 \ge K_{\a, p} C^{1/p} \frac{    2^{-m/p} }{    \frac{2^{-m/p}}{(\log m )^{1/\a}}}
= K_{\a, p} C^{1/p}(\log m )^{1/\a} $$
for $m$ large enough. 
  Hence \eqref{con-1-thm1-alpha}  follows by setting  $C= (p/\alpha +2)^{p/\alpha} K_{\a, p}^{-p}$. 
 \vskip 2 pt
 \noindent
 \emph{Case~\eqref{Int-item-3}}: Let $\Phi_\b(t)=e^{-t^{-1/\b}}$ and note that $\Phi^{-1}_\b(t)=(\log(1/t))^{-\b}$,    $t>0$.   Let  $\b_0>1/\a$ and set  $\Psi=\Phi_{\b_0}$, $\Phi=\Phi_\b$. 
For all $m\geq 1$, 
  \begin{align*}
  y_m\leq y_1= {}&  \int_0^{1} \Big(\log^*\Big(\frac{1}{\e}\Big)\Big)^{1/\a}\,\Psi^{-1}(\dd \e) \\ = {}& \int_0^{\Psi^{-1}(1)} \Big(\log^*\Big(\frac{1}{\Psi(\e)}\Big)\Big)^{1/\a}\, \dd \e\leq K_{\alpha,\beta_0}  \int_0^{\Psi^{-1}(1)} \e^{-1/(\beta_0\alpha )}\, \dd \e<\infty. 
  \end{align*}
%
 Moreover, 
 \begin{align*}
{}& \int_0^{2^{-m}} \Big(\log^*\Big(\frac{2^{-m}}{\e}\Big)\Big)^{1/\a} \,\Psi^{-1}(\dd \e)= 
 \int_0^{2^{-m}} \Big(\log^*\Big(\frac{2^{-m}}{\e}\Big)\Big)^{1/\a} \Big(\log\Big(\frac{1}{\e}\Big)\Big)^{-1-\b_0}\e^{-1}\,\dd \e 
 \\ = {}&  \int_0^{1} \Big(\log^*\Big(\frac{4}{\e}\Big)\Big)^{1/\a} \Big(\log\Big(\frac{2^m}{\e}\Big)\Big)^{-1-\b_0}\e^{-1}\,\dd \e \leq 
K_{\alpha,\b_0} m^{-1-\b_0}.
 \end{align*}
Set $C=1$.  Since $\Phi^{-1}(C 2^{-m})\geq m^{-\b }$ we have that  $x_m\geq m^{1+\b_0-\b}$, and hence  
 \begin{align}\label{lastineq3}
 \sum_{m=1}^\infty 2^m \Phi(M y_m) e^{-x_m^\a}\leq \Phi(M y_1)\sum_{m=1}^\infty \exp\Big(m \log(2) 
  -m^{\alpha(1+\b_0-\b)}\Big) 
<\infty
 \end{align}
since $\a(1+\b_0-\b)>1$.  
  
 \vskip 2 pt
 \noindent
\emph{Case~\eqref{Int-item-4}}: Recall that  $\Phi_{c,r}(t)=e^{-r(\log \frac{1}{t})^c}$ and note that $\Phi_{c,r}^{-1}=\Phi_{\frac{1}{c},r'}$ with $r'=r^{-1/c}$.  
We set $\Psi=\Phi_{c,r}$ and $\Phi=\Phi_{c,v}$ where $v>r$.  
Since $\e\mapsto (\log(1/\e))^{c/\a}$ is slowly varying at zero we have as $n\to \infty$ 
\begin{align*}
y_m={}&  \int_0^{2^{-m}} \Big(\log^*\Big(\frac{1}{\e}\Big)\Big)^{1/\a}\,\Psi^{-1}(\dd \e)
= \int_0^{\Psi^{-1}(2^{-m})} \Big(\log^*\Big(\frac{1}{\Psi(\e)}\Big)\Big)^{1/\a}\,\dd \e \\ = {}& \int_0^{\Phi_{\frac{1}{c},r'}(2^{-m})} \Big(\log\Big(\frac{1}{\e}\Big)\Big)^{\frac{c}{\a}}\,\dd \e\\ 
\sim {}&  \Phi_{\frac{1}{c},r'}(2^{-m})\Big(\log\Big(\frac{1}{\Phi_{\frac{1}{c},r'}(2^{-m})}\Big)\Big)^{\frac{c}{\a}}= (r')^\frac{1}{\a} e^{-r'(m\log 2)^{1/c}} (m\log 2)^{\frac{1}{\a}}.
\end{align*}
    Set $C=1$. Further, as we always have that $x_m\ge  K_\a \Phi^{-1}(C2^{-m})y_m^{-1}$, it follows that
\begin{align*}
x_m\ge  {}&   \frac{1}{2}K_\a \exp\Big[ (r'-v')(m\log 2)^{1/c}\Big] (m\log 2)^{\frac{-1}{\a}}
 \geq 
\exp(\delta m^{1/c})
\end{align*}
for all $m\geq m_0$, for some $m_0\geq 1$, where $\delta=(r'-v')(\log 2)^{1/c}/2$. 
For all $M>0$  we have 
\begin{align*}
 \sum_{m=m_0}^\infty 2^{m}\Phi(M y_m)e^{-x_m^\a}\leq \Phi(M y_1) \sum_{m=m_0}^\infty \exp\Big[ m\log(2)-\exp\Big(\alpha \delta m^{1/c}\Big)\Big]<\infty. 
\end{align*}
\qed

 \begin{remarks} \item 1. In the two first cases of the proof of Theorem~\ref{cor-Int} the     series  in \eqref{con-1-thm1-alpha} converges  slowly,  whereas  in the two last cases,  the above series converges fastly.
 \item 2.  Concerning Case 3,   we observe from \eqref{lastineq3}  that by the assumption made   on $\b$, namely   $\b<1+ \b_0-1/\a$, the convergence of
the series  in
\eqref{con-1-thm1-alpha} is controlled by the sole sequence $(x_m)$, the fact that $(y_m)$ be bounded indeed suffices. If $\b<  \b_0-1/\a$, then 
the convergence of the above  series  
in
\eqref{con-1-thm1-alpha} is already granted by the sequence $(y_m)$. 
\item 3.  Concerning Case 4, let
$0<r''<r'$. Notice, although not used here that   $y_m\le K e^{-r''(m\log 2)^{1/c}}$, and so  
$$ \Phi(M y_m)\le \exp\Big\{ -v\Big(\log \frac{1}{MKe^{-r''(m\log 2)^{1/c}}}\Big)^c \Big\} \le K'\exp\Big\{ -v (r'')^c(m\log 2)  \Big\} , $$ 
where $K,K'$ depend on $r',r'', c, M$. And so  when $0<r<1$, $\Phi(M y_m)  \le K''2^{-\rho m}$ for suitable constants $K''$ and $\rho >1$. Consequently, $
\sum_{m=m_0}^\infty 2^{m}\Phi(M y_m)<\infty$, which of course suffices for concluding. 
\end{remarks}

\begin{remark}\label{Four-con} 
For  $\beta>0$ let  $\Phi_\b(t)=e^{-t^{-1/\beta}}$ for all $t>0$ and $\Phi_{\beta}(0)=0$. Moreover,  for any $r>0$ let $\tilde \Phi_{\b,r}(t)=
\int_0^t \Phi_\beta(s)^r\,ds$,  $t\geq 0$.   Fix an $r>0$. 
From the  Salem--Baernstein theorem, \cite[Theorems~11.8 and 11.13]{Dud-Nor}, it follows  that every continuous periodic function of bounded $\tilde \Phi_{\b,r}$-variation has  uniform convergent  Fourier series  if and only if $\b>1$  (argue as on p.~568 in  \cite{Dud-Nor}).  The same result holds with $\tilde \Phi_{\b,r}$ replaced by $\Phi_\b$, which  follows from the estimate 
\begin{equation*}\label{}
 \tilde  \Phi_{\b,1}(t)\leq \Phi_\b(t)\leq \tilde \Phi_{\beta,\frac{1}{2}} (t)
\end{equation*}
which holds for all  $t$ close enough to zero. 
\end{remark}

\section{Proof of Theorem~\ref{thm-fine-var}.}\label{sec3}

 
Let $n\geq 0$ be a fixed positive integer and $H_n:\R\to\R$ be the $n$-th Hermite polynomial which is defined by 
\begin{equation*}
H_n(x)=\frac{(-1)^n}{n!} e^{x^2/2} \frac{d^n}{dx^n} e^{-x^2/2},\qquad n\geq 1, 
\end{equation*}
and $H_0(x)=1$. For each $f\in L^2(\R)$ set $B(f)=\int_\R f(s)\,\dd B(s)$ and  let $\mathcal H_n$ be the $L^2$-closure of the linear span of 
\begin{equation*}\label{sdkfhsdkhf}
\Big\{H_n(B(f)): \,  f\in L^2(\R),\, \| f\|_{L^2}=1\Big\}.
\end{equation*}
Next choose a real number $a_n>0$ such that 
 \begin{equation}\label{def_a_n-def}
n a_n^{1/n} +  \sum_{k=n+1}^\infty \frac{(k a_n^{2/n} 2/n )^k }{k!} \leq 2. 
 \end{equation}
 All $Y\in \mathcal H_n$ satisfy the following equivalence of moments inequality 
\begin{equation*}\label{skfdhkshf}
\| Y\|_p\leq p^{n/2}\|Y\|_2\qquad \text{for all }p>2,
\end{equation*}  
 cf.\  e.g.\ \cite[Eq.~(4.1)]{Ar-Gi}, and hence for all   $Y\in \mathcal H_n$ with  $\|Y\|_2\leq a_n$ we have  
   \begin{align}\notag
  \E \exp(|Y|^{2/n})
= {}&\sum_{k=0}^n  \frac{1}{k!}\E[ |Y|^{2k/n}] + \sum_{k=n+1}^\infty 
\frac{1}{k!}\E[ |Y|^{2k/n}] 
\\ \label{slfjsfghklh}
\leq {}& n a_n^{1/n} +  \sum_{k=n+1}^\infty \frac{(k a_n^{2/n} 2/n )^k }{k!}\leq 2. 
  \end{align}  
\vskip 2 pt  
  
  Throughout the rest of this section $X= \{X(t)\!: t\in [0,1]\}$ will denote an Hermite process of the form \eqref{eq:74}, and we let $Q_t$ denote the kernel 
    \begin{equation*}\label{def-Q-t}
  Q_t(u_1,\dots,u_m)= c_0\int_0^t \prod_{i=1}^m (v-u_i)_+^{-(1/2+(1-H)/m)}\,\dd v, \qquad u_1,\dots,u_m\in \R. 
  \end{equation*}
     Since $X(t)\in \mathcal  H_{m}$ for all $t\in [0,1]$, \eqref{slfjsfghklh} shows that $X$ satisfies \eqref{eq:673} with $\alpha=2/m$ and $d(s,t)= a_m^{-1} \|X_t- X_s\|_2$ for all $s,t\in [0,1]$, where $a_m$ is given by \eqref{def_a_n-def}.    By  self-similarity and stationary increments  we have that    $d(s,t)=  \tilde c_0 |s-t|^H$ for all $s,t\in [0,1]$ and  a suitable constant  $\tilde c_0$.  We let $\Phi$ be given by \eqref{def-Phi} and $\Xi:\R_+\to\R_+$ denote the function
  \begin{equation*}
\Xi(x)=x^{H}(\log_2^*(1/x))^{m/2}\quad  x>0,\qquad \Xi(0)=0.
\end{equation*}
We note that $\Xi$ is an  asymptotic inverse to $\Phi$ in the sense that 
$\Phi(\Xi(x))\sim \Xi(\Phi(x))\sim x$ as $x\to 0$. 
%
%
%
%

\subsection{Proof of the upper bound in Theorem~\ref{thm-fine-var}}
\label{upper-bound}

In the following we will show the upper bound
\begin{equation}\label{eq:upper}
\mathcal V^*_\Phi(X)\leq \sigma_{m,H}\qquad \text{a.s.}
\end{equation}
To show \eqref{eq:upper} we will  need the following two-sided   Law of the Iterated Logarithm for $X$.

\begin{lemma}\label{two-sided-bound-4}
For each $t>0$ we have with probability one, 
\begin{equation}\label{two-sided-bound}
\lim_{\delta\to0} \Big(\sup_{\substack{u,v\geq 0\\ 0<u+v<\delta}}\, 
\frac{|X(t+u)-X(t-v)|}{\Xi(u+v)}\Big)\leq \sigma_{m,H}^H.
 \end{equation}
 \end{lemma}

Key elements in the proof is a large deviation result by Borell~\cite{Borell} and methods of 
Dudley  and Norvai{\v{s}}a~\cite[Lemma~12.21]{Dud-Nor}.  
\begin{proof}
Fix $t>0$ and consider the process $Z(u,v)=X(t+u)-X(t-v)$. For all $\delta\in (0,1)$ set 
$S(\delta)=\{(u,v)\in \R_+^2: 0<u+v\leq \delta\}$, and note that  $S(\delta)=\delta S(1)$.  By the stationary 
increments and self-similarity of $X$  it follows that $Z$ is self-similar of index $H$, and hence 
\begin{equation}\label{est-234}
 \P\Big(\sup_{(u,v)\in S(\delta)} |Z(u,v)|>z\Big)=\P\Big(\sup_{(u,v)\in S(1)} |Z(u,v)|>z \delta^{-H}\Big).
\end{equation}
By Borell~\cite{Borell}, we have that 
\begin{equation}\label{Borell-result}
t^{-2/m}\log \P\Big(\sup_{(u,v)\in S(1)} | Z(u,v)|\geq t\Big)\to \frac{-1}{2 \varsigma_Z^{2/m}}\qquad \text{as }t\to \infty 
\end{equation} 
where 
\begin{eqnarray*}
 \varsigma_Z &=&\sup_{{\xi\in L^2(\R)\atop \|\xi\|_{L^2(\R)}\leq 1}\atop (u,v)\in S(1)}\Big|\int_{\R^m} \big(Q_{t+u}(u_1,\dots,u_m)-Q_{t-v}(u_1,\dots,u_m)\big) \xi(u_1)\cdots\xi(u_m)\,\dd u_1\cdots \dd u_m\Big|  .
\end{eqnarray*}
Using the scaling property of the  kernel $\{Q_t\!:t\in [0,1]\}$ we obtain that  $\varsigma_Z=2^{-m/2}\sigma_{m,H}^H$.  
Fix $\epsilon\in (0,1/2)$ and choose $\tilde \epsilon>0$ such that $(1-\tilde \epsilon)^2(1+\epsilon)^{2/m}>1$. 
For $n\in \N$ set $\delta_n=e^{-n^{1-\tilde \epsilon}}$,  $\Xi_n=\delta_n^H \varsigma_Z (2\log^*_2(1/\delta_n))^{m/2}$,  $S_n=S(\delta_n)$ and 
\begin{equation*}
E_n=\Big\{\sup_{(u,v)\in S_n} |Z(u,v)|\geq (1+\epsilon) \Xi_n\Big\}. 
\end{equation*}
By \eqref{est-234},  
 \begin{eqnarray*}
\P(E_n)
&= & \P\Big(\sup_{(u,v)\in S(1)} |Z(u,v)|>(1+\epsilon) \Xi_n \delta_n^{-H}\Big)
\\ 
&= & \P\Big(\sup_{(u,v)\in S(1)} |Z(u,v)|>(1+\epsilon)\varsigma_Z\big(2(1-\tilde \epsilon)\log(n)\big)^{m/2}\Big).
\end{eqnarray*}
According to  \eqref{Borell-result}, there exists $T>0$ such that for all $t\geq T$ 
\begin{equation*}\label{est-Z-Borell}
\P\Big(\sup_{(u,v)\in S(1)} |Z(u,v)| \geq t\Big)\leq \exp\Big(-\frac{t^{2/m}(1-\tilde \epsilon)}{2\varsigma_Z^{2/m}}\Big).
\end{equation*}
Hence there  exists a positive  integer $N\geq 1$ such that for all $n\geq N$ we have 
\begin{align*}
 \P(E_n)\leq  \exp\Big(-( 1-\tilde \epsilon)(1+\epsilon)^{2/m}   (1-\tilde \epsilon)\log(n)\Big)
=  n^{-\b}
 \end{align*}
 where $\b:=(1-\tilde \epsilon)^2(1+\epsilon)^{2/m}$. Since $\b>1$, the Borel--Cantelli lemma shows that there exists a measurable set $\Omega_1$ with $\P(\Omega_1)=1$ and for all $\omega\in \Omega_1$ there exists an integer $n_0(\omega)$ such that for all $n\geq n_0(\omega)$ we have that 
 \begin{equation*}
\sup_{(u,v)\in S_n} |Z(u,v,\omega)|\leq  (1+\epsilon) \Xi_n.
 \end{equation*}
For all $\omega\in \Omega_1$, $\delta\leq \delta_n$ where  $n\geq  n_0(\omega)$ we have 
 \begin{eqnarray*}
& &  \sup_{\substack{u,v\geq 0\\ 0<u+v<\delta}}\, 
\frac{|Z(u,v,\omega)|}{\Xi(|u+v|)}  \leq \sup_{j:\, j\geq n} \sup\Big\{ \frac{|Z(u,v,\omega)|}{\Xi(|u+v|)}: (u,v)\in S_j\setminus S_{j+1}\Big\}
\\ & &\qquad  \leq  \sup_{j:\, j\geq n} \sup\Big\{ \frac{|Z(u,v,\omega)|}{\Xi(\delta_{j+1})}: (u,v)\in S_j\setminus S_{j+1}\Big\}
\\ & &\qquad \leq  \sup_{j:\, j\geq n} \frac{(1+\epsilon)\Xi_j}{\Xi(\delta_{j+1})}= (1+\epsilon)\varsigma_Z 2^{m/2} \sup_{j:\, j\geq n}\Big(\Big(\frac{\delta_{j}}{\delta_{j+1}}\Big)^H\Big(\frac{\log_2^*(1/\delta_j)}{\log_2^*(1/\delta_{j+1})}\Big)^{m/2}\Big).
 \end{eqnarray*}
 Since 
 \begin{equation*}
 \frac{\delta_{j}}{\delta_{j+1}}\to 1 \qquad \text{and}\qquad \frac{\log_2^*(1/\delta_j)}{\log_2^*(1/\delta_{j+1})}\to 1\qquad \text{as }j\to\infty
 \end{equation*}
 we can for all $\omega\in \Omega_1$ choose $n\geq n_0(\omega)$ such that  for all $\delta<\delta_n$ we have
 \begin{equation}\label{con-almost-done}
\sup_{\substack{u,v\geq 0\\ 0<u+v<\delta}}\, 
\frac{|Z(u,v,\omega)|}{\Xi(|u+v|)} \leq (1+\epsilon)^2\varsigma_Z 2^{m/2} .
 \end{equation}
Since $\epsilon>0$ was arbitrary chosen and $\sigma_{m,H}^H=\varsigma_Z 2^{m/2}$, \eqref{con-almost-done} implies  \eqref{two-sided-bound}, and the proof is  complete.  
\end{proof}

\begin{proof}[Proof of Theorem~\ref{thm-fine-var}  (the upper bound)]
We will show the upper bound \eqref{eq:upper}. To this aim recall from the beginning  of this section  that $X$ satisfies \eqref{slfjsfghklh} with $d(s,t)= \tilde c_0 | s-t|^H$ and $\alpha= 2/m$. 
 Set $\Psi(t)=t^{1/H}$ for all $t\geq 0$, and recall that $\Phi$ is defined in \eqref{def-Phi}. According to the proof of Theorem~\ref{cor-Int}(i) we know that $\Phi$ and $\Psi$ satisfy \eqref{con-1-thm1-alpha} for a suitable large constant $C>0$, which will be fixed throughout the proof. Furthermore, since  $\mathcal V(\Psi,d)<\infty$ the conditions of Theorem~\ref{thm-alpha} are satisfied.  By definition of $d$ and $\Psi$ above we have that $F(t)=  \tilde c_0^{1/H} t$ for all $t\in [0,1]$.   Set $K=C F(1)$ and let $\e>0$ be a fixed positive number. 
  For  any partition $\pi=\{0=t_0<\dots<t_n=t\}$ of $[0,1]$ let 
 \begin{align*}
 I_1={}& \Big\{i=1,\dots,n: \Phi(|X(t_i,\o)-X(t_{i-1},\o)|)\leq (1+\epsilon) \sigma_{m,H} (t_i-t_{i-1})|\Big\},
\\  I_2={}& \Big\{i=1,\dots,n: (1+\epsilon)\sigma_{m,H} (t_i-t_{i-1}) <\Phi(|X(t_i,\o)-X(t_{i-1},\o)|)\leq K(t_i-t_{i-1})\Big\},
  \\ I_3={}& \Big\{i=1,\dots,n: K(t_i-t_{i-1}) < \Phi(|X(t_i,\o)-X(t_{i-1},\o)|)\Big\}.
 \end{align*}
By the definition of $I_1$ we have the trivial inequality  
 \begin{equation}\label{est-I_1}
 \sum_{i\in I_1} \Phi(|X(t_i)-X(t_{i-1})|)\leq (1+\epsilon)\sigma_{m,H}.
 \end{equation}
 In the following we will show that the sum over $I_2$ and $I_3$ are negligible. 
 
{\em The  $I_2$-sum:}  Set $\beta=(1+\e)^{\frac{1}{1+1/H}}$. For all $\delta>0$  let 
 \begin{align*}
 U_\delta=\Big\{(t,\omega): {}&  |X(t+u,\o)-X(t-v,\o)|)\leq \beta \sigma_{m,H}^H 
  \Xi(u+v) \quad  \text{for all }u,v\in \R_+,\, 0<u+v\leq \delta\Big\}.
 \end{align*}
 By Lemma~\ref{two-sided-bound-4} and an application of Tonelli's theorem there exists a measurable set $\Omega_0\subseteq \Omega$ with $\P(\Omega_0)=1$ such that for all $\omega\in \Omega_0$, $1_{U_{1/n}}(t,\o)\to 1$ for $\lambda$-a.e.\ $t\in (0,1)$, where $\lambda$ denotes the Lebesgue measure.   For $\o\in \O_0$ we have by Lebesgue's theorem,  
 \begin{equation*}
 \lim_{n\to\infty} \lambda(t\in (0,1): (t,\omega)\in U_{1/n})=\lim_{n\to\infty} \int_0^1 1_{\{(t,\omega)\in U_{1/n}\}} \,\dd t=\int_0^1 1\,\dd t=1
 \end{equation*}
and hence  there exists $\delta_0=\delta_0(\omega)$ such that for all $\delta\leq \delta_0$ we have 
 \begin{equation}\label{eq"723}
\lambda\Big(t\in (0,1): (t,\omega)\in U_{\delta}\Big)\geq 1-\epsilon.
 \end{equation}
 For all $r>0$, $\Phi(r \Xi(x))\sim r^{1/H} x$ as $x\to 0$, and hence we 
  may and do assume that $\delta_0$ is chosen such that for all $x\in (0,\delta_0)$,  
  \begin{equation*}
  \Phi(
 \b\sigma_{m,H}^H\Xi(x))\leq \beta \beta^{1/H}\sigma_{m,H}  x=(1+\e)\sigma_{m,H}  x.
  \end{equation*}
  For   $\delta\leq \delta_0$ and  $\pi\in \Pi_\delta$
we have that 
\begin{equation}\label{eq:2323}
\hbox{$(t,\omega)\in U_\delta$ and $t\in [t_{i-1},t_i]$ for some $i=1,\cdots,n  \quad \Rightarrow \quad   i\in I_1$.} 
\end{equation}
By \eqref{eq"723} and  \eqref{eq:2323} we have  
$ \sum_{i\in I_1} (t_i-t_{i-1})\geq 1-\epsilon $ which implies that $\sum_{i\in I_2} (t_i-t_{i-1})\leq \e$. 
Hence 
 \begin{equation}\label{est-I-2-K}
 \sum_{i\in I_2} \Phi(| X(t_i,\o)-X(t_{i-1},\o)|)\leq K \sum_{i\in I_2} (t_i-t_{i-1})\leq K \e. 
 \end{equation}
 
{\em  The $I_3$-sum:} Recalling that   $F(t)=t F(1)$ for all $t\in [0,1]$, and with  $E(\o)$ defined in \eqref{def-of-E-omega} we have   
\begin{align*}
E(\o)={}& \{i=1,\dots,n: \Phi(|X(t_i,\omega)-X(t_{i-1},\omega)|)> C [F(t_i)-F(t_{i-1})]\}\\ ={}& 
 \{i=1,\dots,n: \Phi(|X(t_i,\omega)-X(t_{i-1},\omega)|)> C F(1) (t_i-t_{i-1})\}=I_3.
\end{align*}
Since the assumptions of Theorem~\ref{thm-alpha} are fulfilled,   we have for all  $m_0\in \Z$  with $\max_{1\leq i\leq n}F(t_i)-F(t_{i-1})\leq 2^{-m_0}$ that 
\begin{align}\label{est-I_3}
\sum_{i\in I_3} \Phi(|X(t_i,\o)-X(t_{i-1},\o)|) \leq 9 \sum_{m=m_0}^\infty  \Phi\big( \Theta(\omega) c_d y_m\big)
  Z_{m-2}(\omega)<\infty 
\end{align}
almost surely, cf.\    \eqref{sdhkfs}.  Next we choose   $m_0=m_0(\o)$ such that almost surely 
\begin{equation}\label{est-sum-9}
9 \sum_{m=m_0}^\infty  \Phi\big(\Theta(\omega) y_m\big)
  Z_{m-2}(\omega)\leq \epsilon. 
\end{equation}
Furthermore, choose  $\delta=\delta(\o)>0$ such that $|F(t)-F(s)|\leq 2^{-m_0}$ for all $s,t\in [0,1]$ with $|s-t|\leq \delta$.  
For all $\pi\in \Pi_\delta$ we have     
$\max_{1\leq i\leq n} F(t_i)-F(t_{i-1})\leq 2^{-m_0}$ and hence by \eqref{est-I_3} and \eqref{est-sum-9} 
\begin{equation}
\label{eq-I_3-est} 
\sum_{i\in I_3} \Phi(|X(t_i)-X(t_{i-1})|)\leq  \e 
\end{equation}
with probability one. 

By combining \eqref{est-I_1}, \eqref{est-I-2-K} and \eqref{eq-I_3-est}  we have with probability one
\begin{equation*}
\mathcal V^*_\Phi(X)=\lim_{\delta\to 0}\Big( \sup_{\pi\in \Pi_\delta} v_\Phi(X,\pi)\Big)\leq \epsilon + K\epsilon + (1+\epsilon) \sigma_{m,H},
\end{equation*}
which proves \eqref{eq:upper} since $\e$ was chosen arbitrary.  
\end{proof}

\subsection{Proof of the lower bound in Theorem~\ref{thm-fine-var}}
\label{lower-bound}

We first recall  a result of Marcinkiewicz based on Vitali covering lemma. Let  $f\in  \mathcal B_p$,  then 
\begin{equation}\label{mar1} \int_0^1 \limsup_{h\to 0} \frac{|f(x+h)-f(x)|^p}{|h|}\,\dd x \le \mathcal V_p(f).
\end{equation}
As an extension of this result we have the following lemma (we 
recall that the limiting $\Phi$-variation $\mathcal V^*_\Phi(f)$ is defined in  \eqref{def-lim-Phi}). 

\begin{lemma}\label{lemma-Mar}
Let $\phi,\psi: \R_+\to \R_+$ be continuous, strictly increasing    functions such that $\phi(0)=\psi(0)=0$ and $\lim_{t\to \infty} \phi(t)=\lim_{t\to \infty} \psi(t)=\infty$. Let $H^\psi$ be the   Hausdorff measure with determining function $\psi$. For all  $f:[0,1]\to \R$ measurable we have 
\begin{equation}\label{eq:2341q234}
 H^\psi \Big(x\in [0,1]: \limsup_{h\to 0} \frac{|f(x+h)-f(x)|}{\phi(h)}> 1\Big)\leq  \mathcal V_{\psi\circ \phi^{-1}}^* (f).
\end{equation}
\end{lemma}

Note that  $H^\psi$ coincide with the  Lebesgue measure when $\psi$ is the identity function. 

\begin{proof}
Set $\tilde \Phi=\psi\circ \phi^{-1}$. We use a simple adaptation of Marcinkiewcz's argument, namely a direct application of Vitali's covering lemma for Hausdorff's measures \cite{F}. For all $\delta>0$ let 
\begin{equation*}
E_\delta=
\Big\{[x,x+h]: t\in (0,1),\, h\leq (1-x)\wedge \delta, \   |f(x+h)-f(x)|)>\phi(h)\Big\}.
\end{equation*}
Then $E_\delta$ is a Vitali covering of the set 
\begin{equation*}
A:=\Big\{x\in [0,1]: \limsup_{h\to 0} \frac{|f(x+h)-f(x)|}{\phi(h)}> 1\Big\},
\end{equation*}
 and hence we can for all $\epsilon>0$ pick a finite family of disjoint intervals in $E_\delta$, say $(x_n,x_n+h_n)$ for 
$n=1,\dots,N$,  such that 
$$ \sum_{n=1}^N \psi(h_n)\ge H^\psi (A)-\e .$$
By definition of $E_\delta$ we have that $h_n\leq \delta $ for all $n=1,\dots,N$. Let $\pi$ be a partition in $\Pi_\delta$ which includes the disjoint intervals $(x_n,x_n+h_n)$ 
for $n=1,\dots,N$. By definition of $\Phi$ and $\pi$ 
$$ v_{\tilde \Phi}(f,\pi)\ge  \sum_{n=1}^N \tilde \Phi(|f(x_n+h_n)-f(x_n)|) \ge  \sum_{n=1}^N\psi(h_n)\ge H^\psi (A)-\e,$$
which proves \eqref{eq:2341q234}.
 \end{proof}

 Proposition~\ref{LIL-Hermite} from the appendix has the following corollary which we will use in the proof of the lower bound:
\begin{corollary}\label{cor-LIL}
With probability one, 
 \begin{equation*}
  \limsup_{n\to\infty } \, \frac{|X(1/n)|}{\Xi(1/n)} = \sigma_{m,H}^H.
 \end{equation*}
\end{corollary}

  We are now ready to prove the lower bound in Theorem~\ref{thm-fine-var}:

\begin{proof}[Proof of Theorem~\ref{thm-fine-var}  (the lower bound)]
In the following we will show that
\begin{equation}\label{the-lower-bound023}
\mathcal V^*_\Phi(X)\geq \sigma_{m,H}\qquad \text{a.s.} 
\end{equation}
    Let $\epsilon\in (0,1)$ be fixed. 
Set   $\phi(t)= (1-\epsilon)\sigma_{m,H}^H \Xi(t)$ for $t\geq 0$. By   Corollary~\ref{cor-LIL} and the stationary increments of $X$ we have for all $t\in [0,1]$ that  almost surely 
      \begin{align*}
\limsup_{h\to 0} \frac{|X(t+h)-X(t)|}{\phi(h)}>1.
  \end{align*}
  Hence by Tonelli's theorem, we have  for almost all $\omega\in \Omega$ that 
 \begin{equation}\label{eq:822}
1=\lambda\Big(t\in [0,1]: \limsup_{h\to 0} \frac{|X(t+h,\omega)-X(t,\omega)|}{\phi(h)}>1\Big).
  \end{equation}
   We will use Lemma~\ref{lemma-Mar} with $\psi$ being the identity function, which by  \eqref{eq:822} implies   $\mathcal V^*_{\phi^{-1}}(X)\geq 1$ a.s. Since $\phi^{-1}(t)\sim \sigma_{m,H}^{-1} (1-\e)^{-1/H}\Phi(t)$ as $t\to 0$, we have $\mathcal V_\Phi^*(X)\geq (1-\e)^{1/H}\sigma_{m,H}$ a.s., which implies \eqref{the-lower-bound023} since $\e$ was arbitrary chosen. 
\end{proof}

\subsection{Proof of the remainding parts of Theorem~\ref{thm-fine-var}}
The upper and lower  bounds \eqref{eq:upper}  and \eqref{the-lower-bound023} implies  \eqref{fine-variation}. 
We have that  $\mathcal V_\Phi^*(X)<\infty$ a.s.\ implies $\mathcal V_\Phi(X)<\infty$ a.s.\  since $X$ has bounded sample paths.  The inequality $\mathcal V_\Phi^*(X)\leq \mathcal V_\Phi(X)$ and \eqref{fine-variation} show $\sigma_{m,H}\leq \mathcal V_\Phi(X)<\infty$ a.s.\  To show the last claim let  $\tilde \Phi:\R_+\to\R_+$ be a function satisfying $\tilde \Phi(x) /\Phi(x)\to \infty$ as $x\to 0$.  For each partition $\pi=\{0=t_0<\dots<t_n=1\}$ of $[0,1]$ we have 
   \begin{align}\label{eq-25}
   v_{\tilde \Phi}(X,\pi)= {}& \sum_{i=1}^n \tilde \Phi(|X(t_i)-X(t_{i-1})|)\cr 
  \ge {}&   \Big(\min_{i=1,\dots,n} \frac{\tilde \Phi(|X(t_i)-X(t_{i-1})|)}{\Phi(|X(t_i)-X(t_{i-1})|)}\Big)\sum_{i=1}^n \Phi(|X(t_i)-X(t_{i-1})|).
   \end{align}
For all   $K>0$ choose $\epsilon>0$ such that $\tilde\Phi(x)/\Phi(x)\geq K$ for all $x\in (0,\epsilon]$. By continuity of $X$ we may choose $\delta=\delta(\o) >0$ such that $|X(t)-X(s)|\leq \epsilon$ for all $s,t\in [0,1]$ with $|s-t|\leq \delta$. By \eqref{eq-25}, 
  \begin{align*}
 \sup_{\pi\in \Pi_\delta} v_{\tilde\Phi}(X,\pi)\geq K \sup_{\pi\in \Pi_\delta} v_\Phi(X,\pi)\geq K \sigma_{m,H},
 \end{align*}
   which shows that 
   $\mathcal V_{\tilde \Phi}(X)\geq K \sigma_{m,H}$ for all $K>0$, and hence $\mathcal V_{\tilde \Phi}(X)=\infty$ since $\sigma_{m,H}>0$. This completes the proof of Theorem~\ref{thm-fine-var}. \qed

 \section{Proofs of  Theorem~\ref{t1}, Corollary~\ref{cor-int}
 and   Proposition~\ref{on-ness-con}.}\label{section4}

 \subsection{Proof of Theorem~\ref{t1}} 

 
%
Let $C>0$ be a constant such that \eqref{con-1-thm1-alpha} is satisfied, and notice from the definition of $C$ that it does not depend on the metric $d$.  
 We will use the same notation, and some of the same decompositions, as in the proof of Theorem~\ref{thm-alpha}, in particular, $E(\o), \Theta(\o), c_d, F, y_m$ and $x_m$ are defined as in the proof of Theorem~\ref{thm-alpha}.  Moreover, let $\pi = \{0= t_0<\dots <t_n=1\}$ be a partition of $[0,1]$. Throughout the proof $K$ will denote a constant which does not depend on the metric $d$, but might vary from line to line. 
Set $U:=\{i =1,\dots,n \!: F(t_i)-F(t_{i-1})\leq  1\}$. Since $\sharp (U^c) \leq F(1)$ we have that 
\begin{align*}
 {}&  \sum_{i\in E(\omega)} \Phi\big( |X(t_i,\omega)-X(t_{i-1},\omega)|\big)  \\  {}& \qquad  \leq \sum_{i\in E(\omega)\cap U} \Phi\big( |X(t_i,\omega)-X(t_{i-1},\omega)|\big)  + F(1) \Phi\Big(\sup_{s,t\in [0,1]} |X(t,\o)-X(s,\o)|\Big).
\end{align*}   
Let $m_0$ be the greatest  integer satisfying $F(1)\wedge 1\leq 2^{-m_0}$, and note that $m_0\geq 0$.   By  definition of $U$ and  decomposition \eqref{eq:2424} there exists a measurable set $\Omega_0$ with $\P(\Omega_0)=1$ such that for all $\o\in \Omega_0$,
\begin{equation}
\label{dep2}
  \mathcal   V_\Phi(X(\cdot,\omega))     \le    9 
  \sum_{m=m_0}^\infty  \Phi\Big( \Theta(\omega) c_d y_m \Big) Z_{m-2}(\omega)
+C F(1) + F(1) \Phi\Big(\sup_{s,t\in [0,1]} |X(t)-X(s)|\Big). 
  \end{equation} 
   We have $\|\Theta\|_{\phi_2}\le K$. 
  There is no loss to assume $K\ge 1$. As $ \E e^{(\frac{\Theta }{K})^2}\le 2$, it follows from  Jensen's inequality that
$\E
\Theta
\le K(\log 2)^{1/2} \le K$.
 We have  for $\o\in  A:=\{\Theta\le 4\E
\Theta\}$,  
\begin{equation}\label{eq:6378}
\sum_{m=m_0}^\infty  \Phi\big( \Theta(\omega) c_d y_m \big) Z_{m-2}(\omega)
\leq  K c_d^{p}   \sum_{m=m_0}^\infty 
\Phi\big(y_m\big)  Z_{m-2}(\omega) =:  K c_d^p \mathcal Z(\o), 
  \end{equation}
and   by Tchebycheff's inequality, $\P(A)\ge 3/4$.  
Further, from (\ref{zm1}) 
  \begin{equation*}\label{zm3} 
  \E \mathcal Z=    \sum_{m= m_0}^\infty  \Phi( y_m)  \E Z_{m-2} \leq K F(1) \sum_{m=0}^\infty 2^m \Phi(y_m) e^{-x_m^2}\leq K F(1). 
\end{equation*}
   Let $B:= \{\mathcal Z \le 4\E \mathcal Z\}$. By Tchebycheff's inequality again, $\P(B)\ge 3/4$,  and so $\O_1:=A\cap B$ has
probability larger than $1/2$. By the definition of $\Theta$ in \eqref{theta}, we have that $\sup_{s,t\in [0,1]} |X(t)-X(s)|\leq \Theta \delta (D)$.
Set $\tilde \Psi := \Psi/F(1)$, and hence  $\tilde \Psi^{-1} (x) = \Psi^{-1}(F(1) x)$ for $x\geq 0$. 
Recalling $D\leq \Psi^{-1}(F(1))$, we have  by the assumption \eqref{eq:436728} on $\Psi$ that 
\begin{align*}
\delta(D)\leq {}&  \int_0^{\Psi^{-1}(F(1))} \log^*\Big( \frac{ F(1)}{\Psi(\e)} \Big)^{1/2} \,\dd \e = 
\int_0^{1} \log^*\Big( \frac{ 1}{\e} \Big)^{1/2} \,\tilde \Psi^{-1}(\dd \e) \\ \label{eq:7262}
= {}& F(1) \int_0^{1} \log^*\Big( \frac{ 1}{\e} \Big)^{1/2} (\Psi^{-1})'(F(1) \e)\,\dd \e \leq \Big(K   \int_0^{1} \log^*\Big( \frac{ 1}{\e} \Big)^{1/2} \,\Psi^{-1}(\dd \e) \Big)F(1)^{q}.
\end{align*}
  Thus, for $\o\in \O_1$ we have 
\begin{equation}\label{eq:673657}
\Phi\Big(\sup_{s,t\in [0,1]} |X(t)-X(s)|\Big)\leq \Phi( \Theta \delta(D))
\leq K F(1)^{p q}\Phi(\Theta).
\end{equation}
Recall that $c_d^p = [(4 F(1))\vee 1]^{p/2}\leq K(F(1)^{p/2} +1)$.
By  \eqref{dep2}--\eqref{eq:673657} we have for all 
   $\o\in \O_0\cap \O_1$ that 
\begin{equation*}
\mathcal V_\Phi(X(\cdot,\o))\leq K\Big( F(1)^{p/2+1}    + F(1) + F(1)^{p q}\Big).
\end{equation*}
Hence, by the assumption \eqref{eq:436728} on $\Phi$ we have for all $r>0$ that 
\begin{equation*}
\mathcal V_\Phi(X(\cdot,\o)/r) \leq K r^{-p} \mathcal V_\Phi(X(\cdot,\o))\leq 
K r^{-p} \Big( F(1)^{p/2+1}    + F(1) + F(1)^{p q}\Big), 
\end{equation*}
which by the definition of $\|\cdot \|_\Phi$ shows that 
\begin{equation*}\label{eq:73937}
\|X(\cdot,\o)\|_\Phi \leq K\Big( F(1)^{p/2+1}    + F(1) + F(1)^{p q}\Big)^{1/p}.
\end{equation*}
The estimate \eqref{eq:6378} and  the   strong integrability properties of Gaussian semi-norms (see \cite{F} inequality 0.34) achieves  the proof.
\qed 

\subsection{Proof of Corollary~\ref{cor-int}}
 We apply Theorem~\ref{t1}  with  $\Psi(x)=  x^p[\log_2^*(1/(x\wedge 1))]^{p/2}$ and $\Phi(x)=x^p$   for $x>0$. By the proof of Theorem~\ref{cor-Int}\eqref{Int-item-2} the conditions of Theorem~\ref{thm-alpha} are satisfied. To show  \eqref{eq:436728} we notice that 
 $\Phi(xy)= x^p \Phi(y)$,  and hence the second part of \eqref{eq:436728} is  satisfied. Since  $\Psi$ is strictly increasing and continuous on $[0,\infty)$ and continuous differentiable in $(0,\infty)\setminus \{1\}$, it follows that 
 $\Psi^{-1}$ is absolutely continuous. Moreover, we deduce  that $(\Psi^{-1})'(x)\sim px^{1/p-1}(\log_2^*(1/x))^{-1/2}$ as $x\to 0$,  which implies the existence of two constants $K_1,K_2>0$ such that 
\begin{equation}\label{est-psi'}
K_1 x^{1/p-1} [\log_2^*(1/ (x\wedge 1))]^{-1/2}\leq (\Psi^{-1})'(x)\leq K_{2}  x^{1/p-1} [\log_2^*(1/ (x\wedge 1))]^{-1/2} 
 \end{equation}
 for all $x>0$.  For all $x,y> 0$ we obtain by \eqref{est-psi'} that 
 \begin{align*}
  (\Psi^{-1})'(x y )\leq {}& K (x y)^{1/p-1} [\log_2^*(1/ (xy\wedge 1))]^{-1/2} \\ \leq {}& K 
  x^{1/p-1}  y^{1/p-1}[\log_2^*(1/ (y\wedge 1))]^{-1/2} \leq K x^{1/p-1} (\Psi^{-1})'(y),
 \end{align*}
which shows that $(\Psi^{-1})'$ satisfies the first part of \eqref{eq:436728} with $q=1/p$.  Hence, the corollary follows by Theorem~\ref{t1}. 
\qed

%
%
%
%
 
 \vskip 3 pt
  Before we proving Proposition~\ref{on-ness-con} we note the following: If $\Phi$ satisfies the $\Delta_2$-condition 
  ($\Phi(2x)\leq C\Phi(x)$, $x\geq 0$)
 then  with $c_0: = \log(C)/\log(2)$ we have 
\begin{equation}\label{eq:89214}
\Phi(x)\leq K (1+x^{c_0})\qquad \text{for all }x\geq 0.
\end{equation}
 Indeed, this estimate follows by successive   applications of the  $\Delta_2$-condition.
  
 \subsection{Proof of Proposition~\ref{on-ness-con}}
The assumption  $\mathcal V_\Phi(X)<\infty$ a.s.\  implies that  $\| X \|_\Phi<\infty$ a.s., and  by the strong integrability of Gaussian semi-norms we have  $\E[ e^{\epsilon \| X\|_\Phi^2}]<\infty$ for some $\e>0$. Using \eqref{eq:89214} we  deduce  $\E[\mathcal V_\Phi(X)]<\infty$. Next we note that 
       \begin{align}\label{eq:9462}
   \E[ \mathcal V_\Phi(X)]\geq {}&\sup_{0=t_0<\dots<t_n=1 \atop n\in \N} \sum_{i=1}^n \E[ \Phi(|X({t_i})-X(t_{i-1})|)] \\ \label{eq:9463}\geq {}&  \sup_{0=t_0<\dots<t_n=1 \atop n\in \N} \sum_{i=1}^n\Phi( \E[ |X(t_i)-X(t_{i-1})|]) =  \sup_{0=t_0<\dots<t_n=1 \atop n\in \N} \sum_{i=1}^n\Phi( C d(s,t)) 
  \end{align}
  where the last inequality follows by convexity of $\Phi$ and Jensen's inequality. Note that due to  separability of $X$ it is enough to take supremum over a countable family  of partitions of $[0,1]$ in   \eqref{eq:9462}--\eqref{eq:9463}.  By  \eqref{eq:89214}, we obtain $\mathcal V(\Phi,d)<\infty$ which completes the proof. \qed

\vskip 5pt 
 
Before continuing, let us make a general remark on modulus of continuity and $\Phi$-variation. 
\begin{remark}   A sufficient condition for $f$ to belong to $  \mathcal B_\Phi$ is that there exists $\Psi: \R_+\to \R_+$   increasing and a pseudo-metric $d$ on $T$  such that  
\begin{eqnarray} \label{dqfbq} \qq {\rm (i)}  \ 
\mathcal V(\Psi,d)<\infty, \qquad 
{\rm (ii)}\ \     M_{\Phi, \Psi}(f):=\sup_{s,t\in [0,1]\atop d(s,t )>0} \frac{\Phi\big(|f(s)-f(t )|\big)}{\Psi\big( d(s,t )\big)}<\infty.
\end{eqnarray}
    Choosing for instance  $\Phi(x)= |x|^p$, $\Psi(x)=|x|^p (\log^*|x|)^{p/2}$ shows that $f\in \mathcal B_p$ as soon as 
\begin{eqnarray*} {\rm (a)}\sup_{0=t_0<\dots<t_n=1 \atop n\in \N}\sum_{i=1}^{n-1}d(t_{i+1},t_{i })^p \big(\log^*\frac{1}{d(t_{i+1},t_{i })}\big)^{p/2} <\infty, \quad   {\rm (b)}\sup_{s,t\in [0,1]\atop d(s,t )>0} \frac{ \big|f(s)-f(t )\big| }{d(s,t ) (\log^*\frac{1}{d(s,t )})^{1/2}}<\infty.
\end{eqnarray*}
Letting $f=X(\cdot,\o)$ where $X$ is  Gaussian and  $d(s,t)=\|
X(s) -X(t)\|_2$,  we see that condition (b)   is for instance satisfied  under assumption \eqref{eq:discussJaMo}, which is a weak requirement. However condition (a) although   general, is too strong  compared to  assumption \eqref{eq:JaMo}. As to \eqref{dqfbq}, it  is a consequence of  
\begin{align*} {}& \sum_{i=1}^{n} \Phi\big(|f(t_{i+1})-f(t_i)|\big)=\sum_{i=1}^{n}
\frac{\Phi\big(|f(t_{i+1})-f(t_i)|\big)}{\Psi\big( d(t_{i+1},t_{i })\big)}\Psi\big( d(t_{i+1},t_{i })\big)
\\ {}& \qquad \leq    
\Big(\sup_{s,t\in [0,1]\atop d(s,t )>0} \frac{\Phi\big(|f(s)-f(t )|\big)}{\Psi\big( d(s,t )\big)}\Big)  \sum_{i=1}^{n}\Psi\big( d(t_{i+1},t_{i })\big), 
\end{align*} 
which shows that  $\mathcal V_\Phi(f)\le M_{\Phi, \Psi}(f)\mathcal V(\Psi,d)$.
\end{remark}

%
%


   \appendix
   
  \section{Proof of Theorem~\ref{t2}} \label{appendix-1}
 We prove Theorem \ref{t2} by using the metric entropy method.
  Let 
$$N(\e)= \max(N(T,d,\e ),D/\e) , \qq 0<\e\le D .$$ For   $n=0, 1,\ldots$ let $\e_n=2^{-n }D $, $v_n=12 \e_n (\log N(\e_n))^{1/\a}$,  and let
$\vartheta_n \subset T$ be a sequence of centers of
$d$-balls corresponding to a minimal covering of $T$ of size $\e_n$, $\#\vartheta_n= N(T,d,\e_n)$, and let 
 $\vartheta_0 =\{s_0\}$.  
We first note that 
\begin{eqnarray*}  \sum_{n=1}^\infty v_n&\le & 12 \sum_{n=1}^\infty  \e_n \big(\log N(T,d,\e_n )+ \log  (D/\e_n) )^{1/\a}
\cr &\le &  12\Big( 2\sum_{n=1}^\infty \int_{\e_{n+1}}^{\e_n}   \big(\log N(T,d,\e  )  )^{1/\a}\dd \e +D\sum_{n=1}^\infty  2^{-n}
n^{1/\a}\Big)
\cr &=&  12\Big( 2 \int_{0}^{D/2}   \big(\log N(T,d,\e  )  )^{1/\a}\dd \e +C_\a D \Big)\le C_\a \d(D).\end{eqnarray*}   One can define for
any
$t\in T$,  $\vartheta_n(t) \in  \vartheta_n $ such that
$d(t,
\vartheta_n(t))<2^{-n} D$.   Let $s,t\in T$ such that 
$\e_{k+1}<d(s,t)\le \e_k$ for some $k\ge 0$. Writing 
$$ X(t)= X(\vartheta_k(t)) +\sum_{n=k+1}^\infty Y_n(t)\qq {\rm where}\quad Y_n(t)=X(\vartheta_n(t))-X(\vartheta_{n-1}(t)) ,$$
 we have
  $ |X(s)-X(t)| \le  |X(\vartheta_k(s))-X(\vartheta_k(t))| + 
\sum_{n=k+1 }^\infty
 |Y_n(s)-Y_n(t)| $. Indicate for later use two simple properties.
  \vskip 2 pt  {\bf (i)}  Note that $d(\vartheta_n(t), \vartheta_{n-1}(t))\le d(\vartheta_n(t),t) +d(t,\vartheta_{n-1}(t))\le 3\e_n$. Thus by assumption, 
$\|X(\vartheta_n(t))-X(\vartheta_{n-1}(t))\|_{\phi_\a}\le 3 \e_n $.
  \vskip 3 pt  {\bf (ii)} Next, 
  $d(\vartheta_k(s),\vartheta_k(t))\le d(\vartheta_k(s),t)+ d(s,t)+ d(t,\vartheta_k(t))\le 3\e_k $, so that $d(s,t)\le \e_k$ implies  
that $d(\vartheta_k(s),\vartheta_k(t)) \le 3\e_k $.
\vskip 3 pt 
Now,
  \begin{eqnarray}\label{bbound} \sup_{\e_{k+1}<d(s,t)\le \e_k}  |X(s)-X(t)|  &\le &  \sup_{d(s,t)\le \e_k}
 |X(\vartheta_k(s))-X(\vartheta_k(t))|
   + \sum_{n=k+1 }^\infty
\sup_{s,t\in T} |Y_n(s)-Y_n(t)|
\cr &\le & \sum_{n=k  }^\infty
\xi_n ,
\end{eqnarray} 
where we set 
$$\xi_n = \sup_{d(s,t)\le \e_n}
 |X(\vartheta_n(s))-X(\vartheta_n(t))|+  \sup_{s,t\in T} |Y_n(s)-Y_n(t)| .$$
Note before continuing that by assumption, the series $ \sum_{n= 1 }^\infty
\xi_n $  converges almost surely. 
For, we need the following technical lemma. 
 \begin{lemma}Let $0<\a<\infty$. Let $\xi_1,\dots, \xi_n$, $n\ge 2$,
be  random variables such that for some $\D>0$, $\E \exp\{ \big( {|\xi_i| }/{\D}\big)^\a\}\le 2$, $1\le i\le n$. Then, 
  \begin{eqnarray*}
  \E \exp\Big\{ \Big( \frac{\sup_{1\le i\le n}|\xi_i| }{\D( {2\log n}/{\log 2})^{1/\a}}\Big)^\a\Big\}  \ \le \ 2.
\end{eqnarray*}
  In particular  
\begin{equation}
\label{eq:972} 
  \big\|\sup_{i=1}^n |\xi_i|\big\|_{\phi_\a}\le \big(\sup_{i=1}^n \|\xi_i\|_{\phi_\a}\big)( {2\log n}/{\log 2})^{1/\a}.
\end{equation}
\end{lemma}
\begin{proof} Let  $B= ( {2\log n}/{\log 2})^{1/\a}$.  
Then
 \begin{eqnarray*}\E \exp\Big\{ \Big( \frac{\sup_{1\le i\le n}|\xi_i| }{\D B}\Big)^\a\Big\}&=&\E \exp\Big\{ \Big( \frac{\sup_{1\le i\le n}|\xi_i| }{\D  }\Big)^\a\frac{\log 2}{2\log n}\Big\} 
\cr &\le & \Big( \E \exp\Big\{ \Big( \frac{\sup_{1\le i\le n}|\xi_i| }{\D  }\Big)^\a \Big\} \Big)^{\frac{\log 2}{2\log n}} 
\cr &\le & \Big( \sum_{1\le i\le n} \E \exp\Big\{ \Big( \frac{ |\xi_i| }{\D  }\Big)^\a \Big\} \Big)^{\frac{\log 2}{2\log n}} 
\cr &\le &  (2n  )^{\frac{\log 2}{2\log n}}\ =\  \exp\Big\{\frac{\log (2n)}{2\log n}\log 2\Big\} \ \le \ 2.
\end{eqnarray*}
As to second assertion, 
  by applying the bound previously obtained with $\D= \sup_{i=1}^n \|\xi_i\|_{\phi_\a}$, 
 we get \eqref{eq:972}. 
 \end{proof} 
  Thus by assumption \eqref{eq:673} and Lemma above,
$$\big\| \sup_{d(s,t)\le \e_n}
  |X(\vartheta_n(s))-X(\vartheta_n(t))| \big\|_{\phi_\a}\le    3\e_n  \big(\frac{4\log
 N(T,d,\e_n)}{\log 2} \big)^{1/\a} \le  3\e_n  \big(\frac{4\log
 N( \e_n) }{\log 2} \big)^{1/\a}  , $$
 and  $$ \big\|\sup_{s,t\in T} |Y_n(s)-Y_n(t)| \big\|_{\phi_\a}\le    6\e_n  \big(\frac{4}{\log 2}\log
 N(T,d,\e_n)\big)^{1/\a}.$$ 
 Hence,  $ \sum_{n= 1 }^\infty
\xi_n\in L^{\phi_\a}(\O)$ and the convergence almost sure of the series follows from Beppo-Levi's lemma.  

From (\ref{bbound}), therefore follows that
\begin{eqnarray*} 
\Big\{\exists k\ge 0 : \sup_{\e_{k+1}<d(s,t)\le \e_k}  |X(s)-X(t)|> \sum_{n=k}^\infty v_n\Big\}\subset \Big\{\exists k\ge 0 :
\sum_{n=k  }^\infty
\xi_n> \sum_{n=k}^\infty v_n\Big\}.
\end{eqnarray*} 
 Consequently, 
\begin{eqnarray*} 
\P\Big\{  \sup_{  k\ge 0}\sup_{\e_{k+1}<d(s,t)\le \e_k }  \frac{ |X(s)-X(t)|}{\sum_{n=k}^\infty
v_n} > t  \Big\}&\le &\P\Big\{\exists k\ge 0 :
\sum_{n=k  }^\infty
\xi_n> t\sum_{n=k}^\infty v_k\Big\} 
\cr &\le &\P\big\{\exists j\ge 0 :
 \xi_j> t  v_j\big\}
\cr &\le &  \sum_{j   =0 }^\infty\P \{ 
 \xi_j> t  v_j \}.\end{eqnarray*}
But 
\begin{align*}
 \P \{ 
 \xi_n> t  v_n \} ={}&  \P \Big\{ 
\sup_{d(s,t)\le \e_n}
 |X(\vartheta_n(s))-X(\vartheta_n(t))|+  \sup_{s,t\in T} |Y_n(s)-Y_n(t)|> t  v_n \Big\}
\\ \le {}&  \P \Big\{ 
\sup_{d(s,t)\le \e_n}
 |X(\vartheta_n(s))-X(\vartheta_n(t))| > (t/2)  v_n \Big\}
\\ {}&  +   \P \Big\{ 
  \sup_{s,t\in T} |Y_n(s)-Y_n(t)|> (t/2)  v_n \Big\} .
\end{align*}
  We now assume $t\ge 8^{1/\a}$ and will use the elementary bound which follows from Markov inequality,
 \begin{equation}
 \label{cb}\P\{ |U|\ge u\}\le 1/\phi_\a( {u}/{\|U\|_{\phi_\a} }) 
  \end{equation}
  for all random variables $U$ and $u\geq 0$.   Recalling (ii) and that
$v_n=12
\e_n (\log N(\e_n))^{1/\a}
$,   we get from \eqref{cb}, 
\begin{eqnarray*} 
  & &\P \big\{ 
\sup_{d(s,t)\le \e_n}
 |X(\vartheta_n(s))-X(\vartheta_n(t))| > (t/2)  v_n \big\}
\cr &\le &  \sum_{u,v\in \vartheta_n\atop d(u,v)\le 3\e_n}\P \big\{ 
 \frac{|X(u)-X(v)|}{ d(u,v)}    >  \frac{(t/2)  v_n}{ d(u,v)}   \big\}
\cr &\le &  \sum_{u,v\in \vartheta_n\atop d(u,v)\le 3\e_n}\P \big\{ 
 \frac{|X(u)-X(v)|}{ d(u,v)}    >  \frac{   12t \e_n (\log N(\e_n))^{1/\a}}{6\e_n}   \big\}
  \le   \frac{ \#\{\vartheta_n\}^2  }{\phi_\a
\big( {    t   (\log N(\e_n))^{1/\a}}  \big)}   .\end{eqnarray*} 
  Writing $1+ \phi_\a
 ( {    t   (\log N(\e_n))^{1/\a}}   )= (N(T,d,\e_n)\vee 2^n)^{t^\a-2+2}\ge N(T,d,\e_n)^2 2^{n(t^\a-2)}   $, we observe that 
$N(T,d,\e_n)^2 2^{n(t^\a-2)}-1\ge N(T,d,\e_n)^2 2^{n(t^\a-2)}/2$.
 Thus our estimate produces the bound 
\begin{eqnarray*} 
   \P \big\{ 
\sup_{d(s,t)\le \e_n}
 |X(\vartheta_n(s))-X(\vartheta_n(t))| > (t/2)  v_n \big\}
  \le     \frac{ 2N(T,d,\e_n)^2  }{ N(T,d,\e_n)^2 2^{n(t^\a-2)}    }   \le  2^{-n t^\a/2+1}
.\end{eqnarray*} 
\vskip 1 pt
 Next, let $\theta_n=\vartheta_n\times\vartheta_{n-1} $ and proceed as
follows (recalling (i)),
   \begin{eqnarray*} 
  & &   \P \big\{ 
  \sup_{s,t\in T} |Y_n(s)-Y_n(t)|> (t/2)  v_n \big\}
   \cr &\le &  \sum_{(u,v)\in \theta_n, d(u,v)\le 3\e_n\atop
(u',v')\in \theta_n, d(u',v')\le 3\e_n} \P \big\{ 
   |(X_u-X_v)-(X_{u'}-X_{v'})|> 6 t  \e_n (\log N(\e_n))^{1/\a} \big\}
 \cr &\le &  \sum_{(u,v)\in \theta_n, d(u,v)\le 3\e_n\atop
(u',v')\in \theta_n, d(u',v')\le 3\e_n}\Big\{ \P \big\{ 
   | X_u-X_v  |> 3 t  \e_n (\log N(\e_n))^{1/\a} \big\}
   \cr & &\quad + \P \big\{ 
   | X_{u'}-X_{v'} |> 3 t  \e_n (\log N(\e_n))^{1/\a} \big\}\Big\}
\cr &\le &  2(\#\vartheta_n)^2\sum_{(u,v)\in \theta_n, d(u,v)\le 3\e_n }  \P \big\{ 
   | X_u-X_v  |> 3 t  \e_n (\log N(\e_n))^{1/\a} \big\}
 .\end{eqnarray*}
We have   $\|X(u)-X(v)\|_{\phi_\a}  \le 3\e_n $ for $(u,v)\in \theta_n$.
Therefore, using again   \eqref{cb},
 \begin{eqnarray*} 
    \P \Big\{ 
  \sup_{s,t\in T} |Y_n(s)-Y_n(t)|> (t/2)  v_n \Big\}
    &\le &  \frac{ (\#\vartheta_n)^4  }{\phi_\a
\big( {    t   (\log N(\e_n))^{1/\a}}  \big)}   .\end{eqnarray*} 
 Writing  similarly $1+ \phi_\a
 ( {    t   (\log N(\e_n))^{1/\a}}   )= (N(T,d,\e_n)\vee 2^n)^{t^\a-4+4}\ge N(T,d,\e_n)^4 2^{n(t^\a-4)}   $, we observe that 
$N(T,d,\e_n)^4 2^{n(t^\a-4)}-1\ge N(T,d,\e_n)^4 2^{n(t^\a-4)}/2$. We get here
 \begin{eqnarray*} 
    \P \big\{ 
  \sup_{s,t\in T} |Y_n(s)-Y_n(t)|> (t/2)  v_n \big\}
    &\le &  \frac{ 2 N(T,d,\e_n)^4  }{N(T,d,\e_n)^4 2^{n(t^\a-4)}} \le 2^{-n t^\a/2+1}  ,
 \end{eqnarray*}
since we assumed $t\ge 8^{1/\a}$.   By combining, 
 \begin{eqnarray*} 
 \P \{ 
 \xi_n> t  v_n \}   &\le &   2^{-n t^\a/2+2} .\end{eqnarray*}
Consequently,
\begin{eqnarray*} 
\P\Big\{  \sup_{  k\ge 0}\sup_{\e_{k+1}<d(s,t)\le \e_k }  \frac{ |X(s)-X(t)|}{\sum_{n=k}^\infty
v_n} > t  \Big\}&\le &   2\sum_{n   =0 }^\infty    2^{-n t^\a/2}
 .\end{eqnarray*}
Let $\g =(\log 2)  /6$.
 It follows that 
\begin{eqnarray*} 
& &\int_{ 8^{1/\a}}^\infty e^{\gamma t^\a}\P\Big\{  \sup_{  k\ge 1}\sup_{\e_{k+1}<d(s,t)\le \e_k }  \frac{ |X(s)-X(t)|}{\sum_{n=k}^\infty
v_n} > t  \Big\}\,\dd t \cr &\le &   \sum_{n   =1 }^\infty  \int_{ 8^{1/\a}}^\infty e^{\gamma t^\a - (\log 2)  n ( \frac{t^\a}{2}-1+1)  }
\,\dd t\,\le \,  \sum_{n   =1 }^\infty 2^{-n} \int_{ 8^{1/\a}}^\infty e^{\gamma t^\a - (\log 2)    \frac{t^\a}{4}   } \,\dd t
\cr &= &   \int_{ 8^{1/\a}}^\infty e^{ -( \frac{\log 2 }{12}   )t^\a     } \,\dd
t <\infty.\end{eqnarray*}
Now, note that 
$$\sum_{n=k}^\infty v_n\le 48 \d(2^{-k}D)\le 48 \d(2d(s,t) )\le 96\d( d(s,t) )$$
since $\d$ is concave.
Let $$ \Theta:=  \sup_{ s,t\in T }  \frac{ |X(s)-X(t)|}{\d(d(s,t))} .$$
Thus\begin{eqnarray*} 
 \int_{ 8^{1/\a}}^\infty e^{\gamma t^\a}\P \{  \Theta> t  \}\,\dd t   &= &   \int_{ 8^{1/\a}}^\infty e^{ -( \frac{\log 2 }{12}   )t^\a     } \,\dd
t <\infty.\end{eqnarray*}
This shows that  $ \Theta  \in L^{\phi_\a}(\P)$ 
and moreover that 
$$\|\Theta\|_{\phi_\a}\le  C_\a .  $$

   \section{An Law of the Iterated Logarithm for Hermite processes}\label{appendix-2}
  
  Set 
  \begin{equation*}
  \nu(t)=t^H (1+|\log t|)^{m/2}\qquad t>0,\qquad \nu(0)=0, 
  \end{equation*}
 and let $C_\nu(\R_+)$ be the space of all continuous functions $y:\R_+\to\R$ satisfying  
 \begin{equation*}
 \lim_{t\to\infty } \frac{y(t)}{\nu(t)}=\lim_{t\to 0} \frac{y(t)}{\nu(t)}=0
 \end{equation*}
 equipped with the norm
 \begin{equation*}
 \|y\|_\nu=\sup_{t>0} \frac{|y(t)|}{\nu(t)}. 
 \end{equation*}
 
The following functional Law of the Iterated Logarithm  implies Corollary~\ref{cor-LIL}:

 \begin{proposition}\label{LIL-Hermite}
 For all $n\in \N$ and $t\geq 0$ let 
 \begin{equation}\label{def-X-n}
X_n(t)=\frac{X(t/n)}{n^{-H} (2\log_2(1/n))^{m/2}}. 
 \end{equation}
 With probability one, $X_n\in C_\nu(\R_+)$ for all $n\in \N$. Furthermore, the sequence $\{X_n:n\geq 1\}$ is relative compact in $C_\nu(\R_+)$ and the set of its limits points coincides with $K_Q$, where $K_Q$  is the space of all functions $y:\R_+\to\R$ of the form 
 \begin{equation*}
  y(t)= \int_{\R^m} Q_t(u_1,\dots,u_m)\xi(u_1)\cdots\xi(u_m)\,du_1\cdots du_m,\qquad t\in \R_+
 \end{equation*} 
 where $\xi\in L^2(\R)$ and $\|\xi\|_{L^2(\R)}\leq 1$.
    \end{proposition}

  Theorem~3.1 of Mori and Oodaira~\cite{Mor-Ood} shows Proposition~\ref{LIL-Hermite} where the   processes $X_n$ in  \eqref{def-X-n} are replaced by processes $\tilde X_n$ of the form 
  \begin{equation*}
 \tilde X_n=\frac{X(nt)}{n^H(2\log_2 n)^{m/2}}.
  \end{equation*} 
  In the Gaussian case $m=1$, Proposition~\ref{LIL-Hermite} follows  easily  from Mori and Oodaira's result using  time inversion $X$, that is, 
  $\{X(t):t\geq 0\}\dist \{t^{2H}X(1/t):t\geq 0\}$. However, in the  non-Gaussian case $m\geq 2$ it is not clear to us about such time inversion  holds. 
    To proof  Proposition~\ref{LIL-Hermite} we will use the same approximation 
  of $X$ as is done in the proof of  \cite[Theorem~3.1]{Mor-Ood}. In fact, with the obvious modification, the proof of \cite[Theorem~3.1]{Mor-Ood} will also work for the setting considered in Proposition~\ref{LIL-Hermite}.
 In the following (i)--(iii) we will list which modifications which needs to be done.   (i):  \cite[Lemma~4.1]{Mor-Ood}  holds  with $B_n$ replaced by 
$B_n(t)=B(t/n)/(2 n \log_2n)^{m/2}$.
 This   follows by the time inversion of the Brownian motion and \cite[Lemma~4.1]{Mor-Ood}. (ii):  \cite[Lemma~7.3]{Mor-Ood}  holds with $Z_n^\epsilon(t)$ replaced by $Z_n^\epsilon(t)=X(t/n)/(n^{-H} (2\log_2(1/n))^{m/2})$. This follows by the same arguments as in the proof of \cite[Lemma~7.3]{Mor-Ood}. (iii): In the proof of Theorem~3.1  on the pages 389--390 in \cite{Mor-Ood}  we  replace $X_n^p(t)$ on  the mid of page~389 with $X_n^p(t)=X^p(t/n)/(n^{-H}(2 \log_2 1/n))^{m/2})$, which will prove Proposition~\ref{LIL-Hermite}.

  
  \bibliographystyle{chicago}

\end{document}